\documentclass[10pt,reqno]{amsart}

\usepackage[latin1]{inputenc}
\usepackage[T1]{fontenc}
\usepackage{amsmath,amsfonts,amssymb,amsthm,cases}
\usepackage{geometry}
\usepackage{array,dcolumn,booktabs}
\usepackage[colorlinks=true,linkcolor=blue,citecolor=red]{hyperref} 

\usepackage{etoolbox}
\makeatletter
\patchcmd{\@maketitle}
  {\ifx\@empty\@dedicatory}
  {\ifx\@empty\@date \else {\vskip3ex \centering\footnotesize\@date\par\vskip1ex}\fi
   \ifx\@empty\@dedicatory}
  {}{}
\patchcmd{\@adminfootnotes}
  {\ifx\@empty\@date\else \@footnotetext{\@setdate}\fi}
  {}{}{}
\makeatother

 \author{Victor Vilaça Da Rocha}
\address{Laboratoire de Math\'ematiques Jean Leray, Universit\'e de Nantes, UMR CNRS 6629, \newline
2, rue de la Houssini\`ere, 44322 Nantes Cedex 03, France}
\email{victor.vilaca-da-rocha@univ-nantes.fr}
\date{\today}

\title[Asymptotic behavior of Schrödinger systems on $\mathbb{R}$]
{Asymptotic behavior of solutions to the cubic coupled Schrödinger systems in one space dimension}

\newtheorem{lem}{Lemma}[section]
\newtheorem{pro}[lem]{Proposition}
\newtheorem{theo}[lem]{Theorem}
\newtheorem{rem}[lem]{Remark}

\numberwithin{equation}{section}  %numéroter équation en fonction de section, par exemple eq 2.1

\begin{document}

\begin{abstract}
In this paper, we study a coupled nonlinear Schrödinger system with small initial data in the one dimension Euclidean space.
Such a system appears in the context of the coupling between two different optical waveguides.
We establish an asymptotic nonlinear behavior and a decay estimate for solutions of this system. 
The proof uses a recent work of Kato and Pusateri.
\end{abstract}

\subjclass[2010]{35Q55}
\keywords{Modified Scattering, Nonlinear Schrödinger equation.}
\thanks{Work partially supported by the grant ANA\'E, ANR-13-BS01-0010-03.
The author warmly thanks B. Grébert and L.~Thomann for several useful discussions and
their many suggestions.}

\maketitle

\section{Introduction}

We consider the following system 
\begin{equation}\left\{
\begin{array}{r c l}
i\partial_tu+\partial_{xx}u&=&\left|v\right|^2u, \quad (t,x)\in [1,+\infty)\times\mathbb{R}, \\ 
i\partial_tv+\partial_{xx}v&=&\left|u\right|^2v, \\
u(1,x)=u_1(x), &&\quad v(1,x)=v_1(x).
\end{array}\label{sys}\right.\end{equation}

\subsection{Motivation and background.}
An interesting area within the study of the asymptotic behavior of small solutions of nonlinear dispersive PDEs is the search of nonlinear global dynamics.
In the well-known case of the cubic Schrödinger equation, different methods have been treated by several authors.

On the one hand, the asymptotic behavior depends on the geometry of the spatial domain. 
\begin{itemize}
 \item For Euclidean spaces, we mainly have phase corrections results. 
For example, Kato and Pusateri treat the case $x\in\mathbb{R}$ in \cite{KP} using only tools from the analysis in Fourier space. For general Euclidean spaces, 
we can look at \cite{HN} where Hayashi and Naumkin deal with $x\in\mathbb{R}^n$.
 \item In the case of compact domains, the usual example is the study of the torus $\mathbb{T}^d$. We can for example mention the work of Colliander, Keel, 
Staffilani, Takaoka and Tao who showed some first results in the search of unboundedness of Sobolev norms for small initial data in the spatial space~
$\mathbb{T}^2$ in~\cite{CKSTT}.
\item Combining these two kinds of spaces, we can study the product spaces $\mathbb{R}^n\times \mathbb{T}^d$. 
The first to consider such spaces were Tzvetkov and Visciglia. They obtained some scattering behaviors in~$\mathbb{R}^n\times \mathcal{M}$, where $\mathcal{M}$
is a compact manifold (see \cite{TV}), or in $\mathbb{R}^n\times \mathbb{T}$ where they extended their point of view and studied the well-posedness of their 
equation (see \cite{TV2}).
For the case $\mathbb{R}\times \mathbb{T}^d$, we can cite the work of Hani, Pausader, Tzvetkov and Visciglia, who showed in 2013 (see \cite{HPTV}) a real
growth of Sobolev norms by adding a direction of diffusion to the compact problem. For that purpose, they exhibited a modified scattering of the solution of 
the Schrödinger equation to the solution of the resonant system associated. Contrary to the Euclidean case, this is an effective scattering, and not only
a phase correction.
\end{itemize}

On the other hand, after the geometry of the spatial domain, we can deal with the equation itself. Using the method presented in \cite{HPTV}, similar 
results have been obtained by adding a potential (see \cite{GPT2}), 
\begin{minipage}{\linewidth}
a~harmonic trapping (see \cite{HT}), or by considering different derivatives along 
the Euclidean direction and the periodic one (see \cite{Xu}). In the last mentioned article, Xu exhibits a scattering between the Schrödinger equation~
$i\partial_tU+\Delta_\mathbb{R} U-|\nabla_\mathbb{T}|U=|U|^2U$ in the spatial domain $\mathbb{R}\times\mathbb{T}$ and the cubic Szeg\H{o} equation. \end{minipage}\bigskip

In this article, we focus on a coupled Schrödinger system. Such a system occurs while, for example, looking for the coupling between two different 
optical waveguides, that can be provided by a dual-core single-mode fiber (see \cite{AB}). Other examples are given by two orthogonally polarized components 
traveling at different speeds because of different refractive indices associated with them; or by two distinct pulses with different carrier frequencies but with 
the same polarization. A surprising fact about these examples is that they lead to a system close to the Schrödinger system \eqref{sys}, obtained by interchanging 
the spatial and temporal coordinates.

From our point of view, such a system is interesting. Indeed this coupling effect can provide more nonlinear asymptotic behavior than a single equation.
For example, Grébert, Paturel and Thomann exhibited some beating effects (see \cite{GPT}), that is to say some energy exchanges between different modes of the 
solutions, although they were just dealing with the spatial domain $\mathbb{T}$. Another example of study of Schrödinger systems is given by Kim
in \cite{Kim}. Considering relations on the masses in the equations, the author obtains $L^\infty$ decay of the solution for the spatial domain $\mathbb{R}$.

The goal of this article is to present a self contained proof of nonlinear behavior in the spatial domain~$\mathbb{R}$.
Following the method of Kato and Pusateri in \cite{KP}, we will first prove a $L^\infty$ decay of the solutions. Then, we will highlight an asymptotic behavior, 
in the sense that we will be looking for a pair $(W_f,W_g)$ depending only on the space variable in a space to be determined, such that:
\begin{equation}\left\{\begin{array}{rcl}
\|\hat{w}_f(t,.)-W_f\|&\underset{t\to \infty}{\longrightarrow}&0,\\
\|\hat{w}_g(t,.)-W_g\|&\underset{t\to \infty}{\longrightarrow}&0,
\end{array}\right.\label{W}\end{equation}
where $|\hat{w}_f(t,\xi)|=|\hat{u}(t,\xi)|$ and $|\hat{w}_g(t,\xi)|=|\hat{v}(t,\xi)|$.
This is a result of phase correction scattering, and it is interesting to see what effect the coupling has in this case, to get informations on $W_f$ and $W_g$, 
and to approximate the speed of this convergence.
Moreover, we exhibit an asymptotic formula for large time of the solutions.

\subsection{Notations and norms.}

We define the spatial Fourier transform in Schwartz space, for $\varphi\in S(\mathbb{R})$, by 
$$\mathcal{F}(\varphi)(\xi)=\hat{\varphi}(\xi):=\frac{1}{\sqrt{2\pi}}\displaystyle\int_{\mathbb{R}}e^{-ix\xi}\varphi(x)dx.$$
We use the notation $f\lesssim g$ to denote that there exists a positive constant $M$ such that $f\leq Mg$. 
This notation will be useful, allowing us to avoid dealing with all the constants in the different inequalities.

\noindent When a function depends on two variables $t$ and $x$, we denote by $\varphi(t)$ the function $\varphi(t):x\mapsto\varphi(t,x).$

\noindent Working with small initial data, we can expect the nonlinearity to stay small, and thus the linear dynamics to be dominant. 
This is the reason why it is interesting to consider the profile of a solution, which is the backwards linear evolution
of a solution of the nonlinear equation.
We define $f$ (respectively $g$) the profile of the solution $u$ (respectively $v$) by 
\begin{equation}f(t,x):=e^{-it\partial_{xx}}u(t,x) \qquad\qquad g(t,x):=e^{-it\partial_{xx}}v(t,x).\label{prof}\end{equation}
We will see in Subsection \ref{introprofile} that after considering Duhamel's formula of the solutions, the equations in~$f$ and $g$ are easier to treat
than the ones with $u$ and $v$. This is the main reason why the profiles have been introduced.                                                                                           
 
The norm used will be essential in order to prove the global existence of the solutions.
Let $n\in\mathbb{N}$, we define the spaces $H^{n,0}_x, H^{0,n}_x$ and $L_T^\infty$ corresponding to the following norms:
\begin{itemize}
 \item $\|\varphi\|_{L_x^\infty}:=\displaystyle\sup_{x\in\mathbb{R}}|\varphi(x)|,$
 \item $\|\varphi\|_{H^{n,0}_x}:=\displaystyle\sum_{i=0}^{n}\|\partial^i_x\varphi\|_{L^2_x},$
 \item $\|\varphi\|_{H^{0,n}_x}:=\displaystyle\sum_{i=0}^{n}\|x^i\varphi\|_{L^2_x},$
 \item $\|\varphi\|_{L_T^\infty}:=\displaystyle\sup_{t\in[1,T]}|\varphi(t)|.$
\end{itemize}
With these norms, we follow the method of Kato and Pusateri in \cite{KP} to construct the space 
$$X_T:=\left\{\varphi, \|\varphi\|_{X_T}:=\|\varphi\|_{L^\infty_TL^\infty_x}+\|t^{-\alpha}\varphi\|_{L^\infty_TH^{1,0}_x}
+\|t^{-\alpha}\varphi\|_{L^\infty_TH^{0,2n+1}_x}<\infty\right\},$$
where $\alpha>0$ is small enough.

We want to obtain a $L^\infty$ decay of the solutions. For that purpose, we need to control the behavior of the cubic nonlinearity, which is why the 
terms $\|\varphi\|_{L^\infty_TL^\infty_x}$ and $\|t^{-\alpha}\varphi\|_{L^\infty_TH^{1,0}_x}$ are required in the definition of the norm. 
With these two terms, we can show that we have a unique and global pair of solutions, and that the scattering limits $W_f,W_g$ of the equations \eqref{W} 
are in $L^\infty$. The last term $\|t^{-\alpha}\varphi\|_{L^\infty_TH^{0,2n+1}_x}$ will allow us to show that $W_f,W_g$ are also in $H^{0,n}$.

\subsection{Statement of the results.}

The idea here is to apply the result of \cite{KP} to our coupled Schrödinger system \eqref{sys}. As we can expect, we also obtain a global existence, 
a decay estimate, a scattering result and an asymptotic formula.

\begin{theo}\label{theo}
Fix $n\in\mathbb{N}$ and $0<\nu<\frac14$. Assume $\left\|u_1\right\|_{H^{2n+1,0}_x}+\left\|u_1\right\|_{H^{0,1}_x}+\left\|v_1\right\|_{H^{2n+1,0}_x}+\left\|v_1\right\|_{H^{0,1}_x}\leq\varepsilon$ 
for~$\varepsilon$ small enough.
Let $I=[1,+\infty)$ and set $\mathcal{F}:=\mathcal{C}\left(I;H^{2n+1,0}(I)\cap H^{0,1}(I)\right)$.
Then \eqref{sys} admits a unique pair of solutions~$(u,v)\in\mathcal{F}\times\mathcal{F}$.
The pair of solutions satisfies the following decay estimates:
$$\|u(t)\|_{L^\infty_x}\lesssim\dfrac{1}{t^{\frac{1}{2}}},\qquad\quad\|v(t)\|_{L^\infty_x}\lesssim\dfrac{1}{t^{\frac{1}{2}}}.$$
\begin{minipage}{\linewidth}
Furthermore, there exists a unique pair of functions $(W_f,W_g)\in (L^\infty_\xi\cap H^{0,n}_\xi)\times (L^\infty_\xi\cap H^{0,n}_\xi)$ such that for~$t\geq1$,
\begin{equation*}\begin{cases}
\|\hat{f}(t,.)\exp\left(\dfrac{i}{2\sqrt{2\pi}}\displaystyle\int_1^t\frac{1}{s}\left|\hat{v}(s,.)\right|^2ds\right)-W_f\|_{L^\infty_\xi}&\lesssim t^{-\frac14+\nu},\\
\|\hat{g}(t,.)\exp\left(\dfrac{i}{2\sqrt{2\pi}}\displaystyle\int_1^t\frac{1}{s}\left|\hat{u}(s,.)\right|^2ds\right)-W_g\|_{L^\infty_\xi}&\lesssim t^{-\frac14+\nu},
\end{cases}\end{equation*}
and
\begin{equation*}\begin{cases}
\|\hat{f}(t,.)\exp\left(\dfrac{i}{2\sqrt{2\pi}}\displaystyle\int_1^t\frac{1}{s}\left|\hat{v}(s,.)\right|^2ds\right)-W_f\|_{H^{0,n}_\xi}&\lesssim 
t^{-\frac18+\frac\nu2},\\
\|\hat{g}(t,.)\exp\left(\dfrac{i}{2\sqrt{2\pi}}\displaystyle\int_1^t\frac{1}{s}\left|\hat{u}(s,.)\right|^2ds\right)-W_g\|_{H^{0,n}_\xi}&\lesssim 
t^{-\frac18+\frac\nu2},
\end{cases}\end{equation*}
where $f$ and $g$ are defined in \eqref{prof}.
Finally, we have an asymptotic formula. It exists a unique pair of functions $\Gamma_f,\Gamma_g\in L^\infty_\xi$ such that for large time $t$,
\begin{equation*}\begin{cases}
u(t,x)=\dfrac1{(2it)^{\frac12}}W_f(\dfrac x{2t})
\exp\left(i\dfrac{x^2}{4t}-\dfrac{i}{2\sqrt{2\pi}}(|W_g|^2(\dfrac x{2t})\ln(t)+\Gamma_g(\dfrac x{2t}))\right)+\mathcal{O}(t^{-\frac34+\nu}),\\
v(t,x)=\dfrac1{(2it)^{\frac12}}W_g(\dfrac x{2t})
\exp\left(i\dfrac{x^2}{4t}-\dfrac{i}{2\sqrt{2\pi}}(|W_f|^2(\dfrac x{2t})\ln(t)+\Gamma_f(\dfrac x{2t}))\right)+\mathcal{O}(t^{-\frac34+\nu}).
\end{cases}\end{equation*}
\end{minipage}
\end{theo}

\begin{rem}Let's begin with some remarks about this theorem.
\begin{itemize}
 \item We see that, as we can expect, the role of the coupling appears here as a phase correction term of the scattering result.
 \item By taking $n=0$ in Theorem \ref{theo}, we obtain that $W_f,W_g\in L^\infty_\xi\cap L^2_\xi$ for a small initial data in $H^{1,0}_x\cap H^{0,1}_x$.
 This is the coupled equivalent of the Theorem 1.1 of \cite{KP}. In this article, we are able to get the control of the $H^{0,n}$ norms too.
 \item The decay estimate is not new here. Indeed, Kim already obtained it in \cite{Kim}, but without scattering result on the solutions.
 \item We see that we need a control on the derivatives on $u_1,v_1$ to obtain a control on the $H^{0,n}$ norm of $W_f,W_g$. The reason is that 
 $W_f,W_g$ are the asymptotic behaviors of $w_f$ and $w_g$, defined in \eqref{defw}, which are, up to a phase correction, the Fourier transforms of $u$ and $v$.
 \item They are two losses that can be surprising. First, the loss of speed between the $L^\infty$ behavior and the $H^{0,n}$ one. Secondly,
 the need for $2n+1$ derivatives in $u,v$ to obtain $n$ multiplications in $W_f,W_g$. Both losses are connected to the same phenomenon. The key part of all the 
 estimations will be the good decay in $L^\infty$ of a remainder term $R$ defined in \eqref{Rdef}. To control the $H^{0,n}$ norm of $R$ by using this $L^\infty$ decay, 
 we use the inequality \eqref{perte}, which gives us these losses. These losses may not be sharp, as this control of the $H^{0,n}$ norm may not be optimal.
\end{itemize}
\end{rem}
\noindent The rest of the paper is organized as follows. In Section \ref{sec2}, we will do some preliminary computations and introduce the notion of profile.
In Section \ref{sec3}, we will prove a key proposition. Finally, the proof of the Theorem \ref{theo} will be treated in Section \ref{sec4}.

\section{Preliminaries}\label{sec2}

In this section, we reformulate the problem by considering Duhamel's formula
of the profiles of the solutions. Moreover, we state a proposition, which will be the main ingredient to begin the proof of the Theorem~\ref{theo}. 
Finally, we state and prove a classical estimate of the Schrödinger linear group.

\subsection{Introduction of the profile.}\label{introprofile}

We apply Duhamel's formula to the system \eqref{sys}. If the solution $u$ exists, then $u=\Phi(u)$, where
$$\Phi(u)(t,x)=e^{i(t-1)\partial_{xx}}u_1(x)-i\displaystyle\int_1^te^{i(t-s)\partial_{xx}}|v(s,x)|^2u(s,x)ds.$$
To get the existence of the solution, we have to apply the fixed point theorem. We filter out the linear contribution by introducing the profile here.
For $f$ defined in \eqref{prof}, Duhamel's formula becomes (with the same notation $\Phi$ for the profile),
$$\Phi(f)(t,x)=e^{-i\partial_{xx}}u_1(x)-i\displaystyle\int_1^te^{-is\partial_{xx}}v(s,x)\overline{v}(s,x)u(s,x)ds.$$
Using that $\widehat{uv}=\frac{1}{\sqrt{2\pi}}\hat{u}\ast\hat{v}$, we compute the spatial Fourier transform of $\Phi(f)$ as follows:
$$\begin{array}{rcl}
\widehat{\Phi(f)}(t,\xi)&=&e^{i\xi^2}\hat{u}_1(\xi)-i\displaystyle\int_1^te^{is\xi^2}\widehat{v\overline{v}u}(s,\xi)ds \\
&=&e^{i\xi^2}\hat{u}_1(\xi)-i\displaystyle\int_1^t\displaystyle\int_{\mathbb{R}^2}
e^{is\xi^2}\hat{u}(s,\alpha)\hat{\overline{v}}(s,\eta-\alpha)\hat{v}(s,\xi-\eta)d\eta d\alpha ds\\
&=&e^{i\xi^2}\hat{u}_1(\xi)-i\displaystyle\int_1^t\displaystyle\int_{\mathbb{R}^2}
e^{is\xi^2}\hat{u}(s,\alpha)\overline{\hat{v}}(s,\alpha-\eta)\hat{v}(s,\xi-\eta)d\eta d\alpha ds.
\end{array}$$
Now, we reintroduce $f=e^{-it\partial_{xx}}u$ and $g=e^{-it\partial_{xx}}v$ in the above equation to get:
$$\begin{array}{rcl}
\widehat{\Phi(f)}(t,\xi)&=&e^{i\xi^2}\hat{u}_1(\xi)-i\displaystyle\int_1^t\displaystyle\int_{\mathbb{R}^2}
e^{is\left(\xi^2-\alpha^2+(\alpha-\eta)^2-(\xi-\eta)^2\right)}\hat{f}(s,\alpha)\overline{\hat{g}}(s,\alpha-\eta)\hat{g}(s,\xi-\eta)d\eta d\alpha ds\\
&=&e^{i\xi^2}\hat{u}_1(\xi)-i\displaystyle\int_1^t\displaystyle\int_{\mathbb{R}^2}
e^{2is\eta(\xi-\alpha)}\hat{f}(s,\alpha)\overline{\hat{g}}(s,\alpha-\eta)\hat{g}(s,\xi-\eta)d\eta d\alpha ds.
\end{array}$$
After the change of variables $\alpha\leftrightarrow\xi-\sigma$, we get
\begin{equation}\label{hatf}
\widehat{\Phi(f)}(t,\xi)=e^{i\xi^2}\hat{u}_1(\xi)-\frac{i}{2\pi}\displaystyle\int e^{2is\eta\sigma}
\hat{g}(s,\xi-\sigma)\overline{\hat{g}}(s,\xi-\eta-\sigma)\hat{f}(s,\xi-\eta)d\sigma d\eta ds.
\end{equation}
We define 
\begin{equation}F(s,\eta,\sigma,\xi):=\hat{g}(s,\xi-\sigma)\overline{\hat{g}}(s,\xi-\eta-\sigma)\hat{f}(s,\xi-\eta).\label{defF}\end{equation}
Using the Plancherel formula on \eqref{hatf} we get
$$\widehat{\Phi(f)}(t,\xi)=e^{i\xi^2}\hat{u}_1(\xi)-\frac{i}{2\pi}\displaystyle\int_1^t\displaystyle\int\mathcal{F}_{\eta,\sigma}\left[e^{2is\eta\sigma}\right]
\mathcal{F}^{-1}_{\eta,\sigma}\left[F(s,\eta,\sigma,\xi)\right]d\sigma d\eta ds.$$
Using $\hat{\delta_a}=\frac{1}{\sqrt{2\pi}}e^{-iax},  \widehat{e^{iax}}=\sqrt{2\pi}\delta_a$ and the fact that
$\widehat{f_\lambda}(\xi)=\frac{1}{\lambda}\hat{f}(\frac{\xi}{\lambda})$ where $f_\lambda(x)=f(\lambda x)$, we have
$$\mathcal{F}_{\eta,\sigma}\left[e^{2is\eta\sigma}\right](a,b)=\mathcal{F}_{\eta}\left[\sqrt{2\pi}\delta_{2s\eta}(b)\right](a)
=\frac{1}{2s}e^{-i\frac{ab}{2s}}.$$
Therefore, with the definition of $F$ in \eqref{defF},
$$\begin{array}{rcl}
\widehat{\Phi(f)}(t,\xi)&=&e^{i\xi^2}\hat{u}_1(\xi)-\frac{i}{4\pi}\displaystyle\int_1^t\displaystyle\int\frac{1}{s}e^{-i\frac{ab}{2s}}
\mathcal{F}^{-1}_{\eta,\sigma}\left[F(s,\eta,\sigma,\xi)\right](a,b)dadbds\\
&=&e^{i\xi^2}\hat{u}_1(\xi)-\frac{i}{4\pi}\displaystyle\int_1^t\frac{1}{s}\displaystyle\int\mathcal{F}^{-1}_{\eta,\sigma}
\left[F(s,\eta,\sigma,\xi)\right](a,b)dadbds \\
&&-\frac{i}{4\pi}\displaystyle\int_1^t\frac{1}{s}\displaystyle\int\left(e^{-i\frac{ab}{2s}}-1\right)
\mathcal{F}^{-1}_{\eta,\sigma}\left[F(s,\eta,\sigma,\xi)\right](a,b)dadbds\\
&=&e^{i\xi^2}\hat{u}_1(\xi)-\frac{i\sqrt{2\pi}}{4\pi}\displaystyle\int_1^t\frac{1}{s}\left|\hat{g}(s,\xi)\right|^2\hat{f}(s,\xi)ds
+\displaystyle\int_1^tR(s,\xi)ds,
\end{array}$$
where \begin{equation}R(s,\xi)=R(\hat{g},\hat{g},\hat{f})(s,\xi):=-\frac{i}{4\pi s}\displaystyle\int\left(e^{-i\frac{ab}{2s}}-1\right)
\mathcal{F}^{-1}_{\eta,\sigma}\left[F(s,\eta,\sigma,\xi)\right](a,b)dadb.\label{Rdef}\end{equation}
Thus, we obtain
\begin{equation}\partial_t\widehat{\Phi(f)}(t,\xi)=-\frac{i}{2\sqrt{2\pi}t}\left|\hat{g}(t,\xi)\right|^2\hat{f}(t,\xi)+R(t,\xi).\label{prob}\end{equation}
The idea is to use the good behavior of $R$, which has a good decay in time. To take advantage of this decay, we introduce a last change of variable.
We set
\begin{equation}\begin{array}{rcl rcl}
B_g(t,\xi)&=&\exp\left(\dfrac{i}{2\sqrt{2\pi}}\displaystyle\int_1^t\frac{1}{s}\left|\hat{g}(s,\xi)\right|^2ds\right),
&B_f(t,\xi)=&\exp\left(\dfrac{i}{2\sqrt{2\pi}}\displaystyle\int_1^t\frac{1}{s}\left|\hat{f}(s,\xi)\right|^2ds\right),\\
\hat{w}_f(t,\xi)&=&\hat{f}(t,\xi)B_g(t,\xi), &\hat{w}_g(t,\xi)=&\hat{g}(t,\xi)B_f(t,\xi).
\end{array}\label{defw}\end{equation}
The main property of this change of variables is the conservation of the modulus of $f$ : $|\hat{w}_f(t,\xi)|=|\hat{f}(t,\xi)|$.
Thus, we have $\hat{f}=\widehat{\Phi(f)}$ if and only if $\hat{w}_f$ is a solution of the following differential equation: 
$$\partial_t\hat{w}_f(t,\xi)=B_g(t,\xi)R(t,\xi).$$
Applying Duhamel's formula a second time, we define the Duhamel form for $(\hat{w}_f,\hat{w}_g)$ (still using the same notation $\Phi$) by:
\begin{equation}\begin{cases}
\Phi_1(\hat{w}_f,\hat{w}_g)(t,\xi)&=\hat{w}_f^1(\xi)+\displaystyle\int_1^tB_g(s,\xi)R(\hat{g},\hat{g},\hat{f})(s,\xi)ds, \\
\Phi_2(\hat{w}_f,\hat{w}_g)(t,\xi)&=\hat{w}_g^1(\xi)+\displaystyle\int_1^tB_f(s,\xi)R(\hat{f},\hat{f},\hat{g})(s,\xi)ds, 
\end{cases}\label{Phi}\end{equation}
where $\hat{w}_f^1(\xi)=\hat{w}_f(1,\xi)$ and $\hat{w}_g^1(\xi)=\hat{w}_g(1,\xi)$.
They are the equations that we will treat to get the existence of the solutions. For that purpose, we are now able to state the key result of the article.
\begin{pro}\label{pro}
Assume $\|\hat{w}_f^1\|_{H^{1,0}_\xi}+\|\hat{w}_f^1\|_{H^{0,2n+1}_\xi}+\|\hat{w}_g^1\|_{H^{1,0}_\xi}+\|\hat{w}_g^1\|_{H^{0,2n+1}_\xi}\leq\varepsilon$,
then there exists a positive constant M such that
\begin{equation}\begin{array}{rcl}
\left\|\Phi_1(\hat{w}_f,\hat{w}_g)\right\|_{X_T}&\leq&\varepsilon+M\left\|\hat{w}_f\right\|_{X_T}\left\|\hat{w}_g\right\|_{X_T}^2(1+\left\|\hat{w}_f\right\|_{X_T}^4)
(1+\left\|\hat{w}_g\right\|_{X_T}^4),\\
\left\|\Phi_2(\hat{w}_f,\hat{w}_g)\right\|_{X_T}&\leq&\varepsilon+M\left\|\hat{w}_g\right\|_{X_T}\left\|\hat{w}_f\right\|_{X_T}^2(1+\left\|\hat{w}_g\right\|_{X_T}^4)
(1+\left\|\hat{w}_f\right\|_{X_T}^4).
\label{prop}\end{array}\end{equation}
\end{pro}

\begin{rem}
There are two main differences between this estimation and the Proposition 1.3 of Kato and Pusateri in \cite{KP}.
First, we are dealing here with a result about $w$, and not about $u$ directly. The reason is in the equation \eqref{prob}.
As long as we have no proof of the existence of $f$, we cannot take $\Phi(f)=f$ and do a legit change of variables in this equation.
But the existence of $f$ is equivalent to the existence of $w$, which is why we focus on $w$ here.
Secondly, the right-hand side in \cite{KP} is a cubic term. The apparition of $(1+\left\|\hat{w}_f\right\|_{X_T}^4)(1+\left\|\hat{w}_g\right\|_{X_T}^4)$ terms in 
this paper is due to the choice of the $X_T$-norm. In particular, the terms $\|\partial_{\xi}\hat{f}(s)\|_{L^2_\xi}$ and 
$\|\partial_{\xi}\hat{g}(s)\|_{L^2_\xi}$ are not controlled by the definition of the $X_T$-norm. In fact, as soon as we consider small solutions,
the right-hand side is very close to a cubic term.
\end{rem}

\subsection{Properties of the linear Schrödinger group.}

Before dealing with the computations on the $X_T$ norm, we introduce the following result of independent interest, about the linear Schrödinger group.
This lemma will be used in both proofs of Proposition \ref{pro} and Theorem \ref{theo}.

\begin{lem}\label{lem}
 For all $\varphi\in\mathcal S(\mathbb R)$ and for all $0<\beta<\frac{1}{4}$, we have
 $$\|e^{it\partial_{xx}}\varphi\|_{L^\infty_x}\lesssim\dfrac{1}{|t|^{\frac{1}{2}}}\|\hat{\varphi}\|_{L^\infty_\xi}
 +\dfrac{1}{|t|^{\frac{1}{2}+\beta}}\|\hat{\varphi}\|_{H^{1,0}_\xi}.$$ 
\end{lem}

\begin{proof}
The following proof is based on the proof of the more general Lemma in \cite{HN}.
The idea is to write~$e^{it\partial_{xx}}\varphi$ as an integral with a kernel.
% Thus, we look for $K_t$ such that $e^{it\partial_{xx}}\varphi=K_t\ast \varphi$. 
% Using $\widehat{f\ast g}=\sqrt{2\pi}\hat{f}\hat{g}$, we have 
% $$\widehat{\left(e^{it\partial_{xx}}\varphi\right)}(\xi)=e^{-it\xi^2}\hat{\varphi}(\xi) \qquad 
% \text{and} \qquad \widehat{\left(K_t\ast \varphi\right)}(\xi)=\sqrt{2\pi}\hat{K_t}(\xi)\hat{\varphi}(\xi).$$
% Therefore, we look for some gaussian for $K_t:x\mapsto \alpha_te^{-\beta_tx^2}$, such that $\hat{K_t}(\xi)=\frac{1}{\sqrt{2\pi}}e^{-it\xi^2}$.
% Taking the derivative with respect to $x$, we have $K_t'(x)=-2\beta_txK_t(x)$, which becomes $i\xi\hat{K_t}(\xi)=-2i\beta_t\hat{K_t}'(\xi)$ in Fourier space. 
% So, we have $\hat{K_t}(\xi)=\hat{K_t}(0)e^{-\frac{\xi^2}{4\beta_t}}=\dfrac{\alpha_t}{(2\beta_t)^{1/2}}e^{-\frac{\xi^2}{4\beta_t}}$. 
% Thus, we take $\beta_t=\frac{1}{4it}$ and $\alpha_t=(4i\pi t)^{-1/2}$.
% We obtain 
% \begin{equation*} 
% K_t(x)=(4i\pi t)^{-1/2}e^{-\frac{x^2}{4it}}.
% \end{equation*}
Recall that for 
\begin{equation*} 
K_t(x)=(4i\pi t)^{-\frac12}e^{-\frac{x^2}{4it}},
\end{equation*}
we have $e^{it\partial_{xx}}\varphi=K_t\ast \varphi$.
Thus,
$$\begin{array}{rcl}
e^{it\partial_{xx}}\varphi(x)&=&\displaystyle\int_{\mathbb{R}}\varphi(y)K_t(x-y)dy \\
&=&(4i\pi t)^{-\frac12}\displaystyle\int_{\mathbb{R}}\varphi(y)e^{-\frac{(x-y)^2}{4it}}dy \\
&=&(4i\pi t)^{-\frac12}e^{i\frac{x^2}{4t}}\displaystyle\int_{\mathbb{R}}\varphi(y)e^{i\frac{y^2}{4t}}e^{-i\frac{xy}{2t}}dy \\
&=&(4i\pi t)^{-\frac12}e^{i\frac{x^2}{4t}}\left(\displaystyle\int_{\mathbb{R}}\varphi(y)e^{-i\frac{xy}{2t}}dy
+\displaystyle\int_{\mathbb{R}}\varphi(y)(e^{i\frac{y^2}{4t}}-1)e^{-i\frac{xy}{2t}}dy\right).\end{array}$$
We define $A_\varphi$ such that
\begin{equation}
e^{it\partial_{xx}}\varphi(x):=(2it)^{-\frac12}e^{i\frac{x^2}{4t}}\hat{\varphi}(\frac{x}{2t})+A_\varphi(t,x).
\label{defA}\end{equation}
The first part of the sum gives us the $L^\infty$ part of the Lemma. We now deal with the term $A_\varphi$ by using the fact that 
(with $0<\beta<\frac{1}{4}$):
\begin{equation}\left|e^{ix}-1\right|=2\left|\sin(\frac x2)\right|\lesssim x^{\beta}.\label{sin}\end{equation}
This trivial inequality gives us for $A$
\begin{equation*}\begin{array}{rcl}
\left\|A_\varphi(t)\right\|_{L^\infty_x}&\lesssim&\left|t\right|^{-\frac12}\left\|\left|x\right|^{2\beta}\left|t\right|^{-\beta}\varphi\right\|_{L^1_x}\\
&\lesssim&\left|t\right|^{-\frac12-\beta}\left\|\left|x\right|^{2\beta}\left<x\right>\left<x\right>^{-1}\varphi\right\|_{L^1_x} \\
&\lesssim&\left|t\right|^{-\frac12-\beta}\left\|\left|x\right|^{2\beta}\left<x\right>^{-1}\right\|_{L^2_x}\left\|\left<x\right>\varphi\right\|_{L^2_x} \\
&\lesssim&\left|t\right|^{-\frac12-\beta}\|\hat{\varphi}\|_{H^{1,0}_\xi}.
\end{array}\end{equation*}
This concludes the proof of the lemma.
\end{proof}

\section{Proof of the Proposition \ref{pro}}\label{sec3}

In this section, we prove the Proposition \ref{pro}, following the proof of the similar proposition in \cite{KP}.
We want to deal with the $X_T$ norms of $\hat{w}_f$ and $\hat{w}_g$. Here we write the computations for $\hat{w}_f$, obviously, the results are the same mutatis
mutandis for $\hat{w}_g$. We split the proof in three parts because of the several estimations needed in the $X_T$ norm.

\subsection{The decay of \texorpdfstring{$R$}{Lg}, one \texorpdfstring{$L^{\infty}$}{Lg} estimate.}
We first focus on $R$. We recall that
\begin{equation*}R(s,\xi)=R(\hat{g},\hat{g},\hat{f})(s,\xi)=-\frac{i}{4\pi s}\displaystyle\int\left(e^{-i\frac{ab}{2s}}-1\right)
\mathcal{F}^{-1}_{\eta,\sigma}\left[F(s,\eta,\sigma,\xi)\right](a,b)dadb,\end{equation*}
with 
\begin{equation*}F(s,\eta,\sigma,\xi):=\hat{g}(s,\xi-\sigma)\overline{\hat{g}}(s,\xi-\eta-\sigma)\hat{f}(s,\xi-\eta).\end{equation*}
We have the following lemma:
\begin{lem}\label{lemR}
For all $0<3\alpha<\delta<\frac{1}{4}$ we have 
$$\left|R(s,\xi)\right|\lesssim s^{-1-\delta}\left\|\hat{g}\right\|_{H^{1,0}_\xi}^2\|\hat{f}\|_{H^{1,0}_\xi}.$$
\end{lem}
\begin{rem}\label{rem4}
This lemma is the key of the estimates we will deal with.
\begin{itemize}
 \item The important result here is this small $\delta$ that ensures the integrability of 
$\left\|R(s,.)\right\|_{L^\infty_\xi}$. Moreover, this provides a good decay that will erase the behavior of the other parts, because $3\alpha<\delta$,
and ensures the existence of the solutions of our system \eqref{sys}.
 \item We see here why we choose the time $1$ for initial data in the system \eqref{sys}. Indeed, here is the main problem with the integral in $0$. 
Almost all the $\frac{1}{t}$ terms can be controlled in $0$ by changing the coefficients in the $X_T$ norm (for example by changing the $t$ in $1+t$),
but not this $\frac{1}{s^{1+\delta}}$ here. This is a consequence of the behavior of the solution of the linear Schrödinger equation near $0$.
Our purpose here is to highlight an asymptotic behavior, this is the reason why we just avoid this problem by taking a strictly positive time for the 
initial data.
 \item This lemma is satisfied for all $0<3\alpha<\delta<\frac{1}{4}$. The goal after the lemma is to take $\alpha$ small enough, and $\delta$ as big as
possible. Keeping this in mind, in the Subsection \ref{subfin}, it will be natural to consider that we can choose $\alpha$ and $\delta$ such that 
$0<4\alpha<\delta<\frac{1}{4}$.
\end{itemize}
\end{rem}

\begin{proof}
By definition we have 
$$\left|R(s,\xi)\right|\lesssim\frac{1}{s}\displaystyle\int\left|e^{-i\frac{ab}{2s}}-1\right|
\left|\mathcal{F}^{-1}_{\eta,\sigma}\left[F(s,\eta,\sigma,\xi)\right](a,b)\right|dadb.$$
Once again we use \eqref{sin} with $0<\delta<\frac{1}{4}$ to obtain 
$$\left|R(s,\xi)\right|\lesssim s^{-1-\delta}\displaystyle\int\left|a\right|^{\delta}\left|b\right|^{\delta}
\left|\mathcal{F}^{-1}_{\eta,\sigma}\left[F(s,\eta,\sigma,\xi)\right](a,b)\right|dadb.$$
Let's focus on $\mathcal{F}^{-1}_{\eta,\sigma}F$,
$$\begin{array}{rcl}
\mathcal{F}^{-1}_{\eta,\sigma}\left[F(s,\eta,\sigma,\xi)\right](a,b)&=&
\mathcal{F}^{-1}_{\eta,\sigma}\left[\hat{g}(s,\xi-\sigma)\overline{\hat{g}}(s,\xi-\eta-\sigma)\hat{f}(s,\xi-\eta)\right](a,b)\\
&=&\mathcal{F}^{-1}_{\eta}\left[\hat{f}(s,\xi-\eta)\mathcal{F}^{-1}_{\sigma}
\left(\hat{g}(s,\xi-\sigma)\overline{\hat{g}}(s,\xi-\eta-\sigma)\right)(b)\right](a).
\end{array}$$
But $\mathcal{F}^{-1}(fg)=\frac{1}{\sqrt{2\pi}}\mathcal{F}^{-1}(f)\ast\mathcal{F}^{-1}(g)$, therefore,
$$\begin{array}{rcl}
\mathcal{F}^{-1}_{\eta,\sigma}\left[F(s,\eta,\sigma,\xi)\right](a,b)&=&\mathcal{F}^{-1}_{\eta}\left[\hat{f}(s,\xi-\eta)
\frac{1}{\sqrt{2\pi}}\mathcal{F}^{-1}_{\sigma}\hat{g}(s,\xi-\sigma)\ast\mathcal{F}^{-1}_{\sigma}\overline{\hat{g}}(s,\xi-\eta-\sigma)(b)\right](a).
\end{array}$$
Moreover, $\mathcal{F}_\sigma^{-1}[\hat{g}(\xi-\sigma)](b)=e^{i\xi y}g(-b)$ and
$\mathcal{F}_\sigma^{-1}[\overline{\hat{g}(\xi-\sigma)}](b)=e^{i\xi b}\overline{g}(b)$,
$$\begin{array}{rcl}
\mathcal{F}^{-1}_{\eta,\sigma}\left[F(s,\eta,\sigma,\xi)\right](a,b)&=&\frac{1}{\sqrt{2\pi}}\mathcal{F}^{-1}_{\eta}\left[\hat{f}(s,\xi-\eta)
\left(e^{ib\xi}g(s,-b)\ast_be^{ib(\xi-\eta)}\overline{g}(s,b)\right)(b)\right](a)\\
&=&\frac{1}{\sqrt{2\pi}}\mathcal{F}^{-1}_{\eta}\left[\hat{f}(s,\xi-\eta)
\displaystyle\int_{\mathbb{R}}e^{iy\xi}g(s,-y)e^{i(b-y)(\xi-\eta)}\overline{g}(s,b-y)dy\right](a)\\
&=&\frac{1}{\sqrt{2\pi}}\displaystyle\int_{\mathbb{R}}e^{ib\xi}g(s,-y)\overline{g}(s,b-y)
\mathcal{F}^{-1}_{\eta}\left(e^{-i(b-y)\eta}\hat{f}(s,\xi-\eta)\right)(a)dy.
\end{array}$$
But $\mathcal{F}_\eta^{-1}[e^{-ib\eta}\hat{f}(\xi-\eta)](a)=e^{i\xi(a-b)}f(b-a)$,
$$\begin{array}{rcl}
\mathcal{F}^{-1}_{\eta,\sigma}\left[F(s,\eta,\sigma,\xi)\right](a,b)&=&
\frac{1}{\sqrt{2\pi}}\displaystyle\int_{\mathbb{R}}e^{ib\xi}g(s,-y)\overline{g}(s,b-y)e^{i\xi(y+a-b)}f(s,b-y-a)dy.
\end{array}$$
Finally, with the change of variables $x\leftrightarrow(b-y)$ we obtain
$$\begin{array}{rcl}
\mathcal{F}^{-1}_{\eta,\sigma}\left[F(s,\eta,\sigma,\xi)\right](a,b)&=&\frac{1}{\sqrt{2\pi}}\displaystyle\int_{\mathbb{R}}e^{ib\xi}g(s,x-b)\overline{g}(s,x)
e^{i\xi(a-x)}f(s,x-a)dx\\
&=&\frac{1}{\sqrt{2\pi}}e^{i(a+b)\xi}\displaystyle\int_{\mathbb{R}}e^{-ix\xi}g(s,x-b)\overline{g}(s,x)f(s,x-a)dx.\\
\end{array}$$
Therefore, back on $R$, we have 
$$\left|R(s,\xi)\right|\lesssim s^{-1-\delta}\displaystyle\int\left|a\right|^{\delta}\left|b\right|^{\delta}
\left|g(s,x-b)\right|\left|g(s,x)\right|\left|f(s,x-a)\right|dxdadb.$$
Using that 
$\left|a\right|^\delta\left|b\right|^\delta\lesssim(\left|x-a\right|^\delta+\left|x\right|^\delta)(\left|x-b\right|^\delta+\left|b\right|^\delta)$,
we obtain 
$$\begin{array}{rcl}
\left|R(s,\xi)\right|&\lesssim& s^{-1-\delta}\displaystyle\int\left|x-b\right|^\delta\left|g(s,x-b)\right|
\left|g(s,x)\right|\left|x-a\right|^\delta\left|f(s,x-a)\right|dxdadb\\
&&+s^{-1-\delta}\displaystyle\int\left|g(s,x-b)\right|
\left|x\right|^\delta\left|g(s,x)\right|\left|x-a\right|^\delta\left|f(s,x-a)\right|dxdadb\\
&&+s^{-1-\delta}\displaystyle\int\left|x-b\right|^\delta\left|g(s,x-b)\right|
\left|x\right|^\delta\left|g(s,x)\right|\left|f(s,x-a)\right|dxdadb\\
&&+s^{-1-\delta}\displaystyle\int\left|g(s,x-b)\right|\left|x\right|^{2\delta}\left|g(s,x)\right|\left|f(s,x-a)\right|dxdadb\\
&\lesssim&s^{-1-\delta}\left(\|\left|x\right|^\delta g\|_{L^1_x}\left\|g\right\|_{L^1_x}\|\left|x\right|^\delta f\|_{L^1_x}+
\left\|g\right\|_{L^1_x}\|\left|x\right|^\delta g\|_{L^1_x}\|\left|x\right|^\delta f\|_{L^1_x}\right.\\
&+&\left.\|\left|x\right|^\delta g\|_{L^1_x}\|\left|x\right|^\delta g\|_{L^1_x}\left\|f\right\|_{L^1_x}+
\|\left|x\right|^{2\delta} g\|_{L^1_x}\left\|g\right\|_{L^1_x}\left\|g\right\|_{L^1_x}\right).
\end{array}$$
Using the Cauchy-Schwarz inequality and the fact that $\delta<\frac14$, we get
$$\left\|\left<x\right>^{2\delta}f\right\|_{L^1_x}\leq\left\|\left<x\right>^{2\delta}\left<x\right>^{-1}\right\|_{L^2_x}
\left\|\left<x\right>f\right\|_{L^2_x}\lesssim\|\hat{f}\|_{H^{1,0}_\xi}.$$
Therefore,  
$$\left|R(s,\xi)\right|\lesssim s^{-1-\delta}\|\hat{f}\|_{H^{1,0}_\xi}\|\hat{g}\|_{H^{1,0}_\xi}^2.\\$$
\end{proof}

\subsection{Some \texorpdfstring{$L^2$}{Lg} estimates.}
By the previous estimation, we see that we first need to control the $\|\hat{f}\|_{H^{1,0}_\xi}$ term, which will occur in a lot of
our computations. By definition,
$$\begin{array}{rcl}
\|\hat{f}(t)\|_{H^{1,0}_\xi}&=&\|\hat{f}(t)\|_{L^2_\xi}+\|\partial_\xi\hat{f}(t)\|_{L^2_\xi}\\
&=&\|\hat{w}_f(t)\|_{L^2_\xi}+\|\partial_\xi(\hat{w}_f\overline{B_g})(t)\|_{L^2_\xi}\\
&\leq&\|\hat{w}_f(t)\|_{L^2_\xi}+\|\partial_\xi\hat{w}_f(t)\|_{L^2_\xi}+\|\hat{w}_f(t)\|_{L^\infty_\xi}\|\partial_\xi B_g(t)\|_{L^2_\xi}\\
&\leq&\|\hat{w}_f(t)\|_{H^{1,0}_\xi}+\|\hat{w}_f(t)\|_{L^\infty_\xi}\displaystyle\int_1^ts^{-1}\|\partial_\xi\hat{w}_g.\hat{w}_g(s)\|_{L^2_\xi}ds.
\end{array}$$
Thus, using the definition of the $X_T$-norm, we obtain
\begin{equation}\|\hat{f}(t)\|_{H^{1,0}_\xi}\lesssim t^\alpha\left\|\hat{w}_f\right\|_{X_T}(1+\left\|\hat{w}_g\right\|_{X_T}^2).\label{fh10}\end{equation}
We now have all the tools for studying $\|R(s)\|_{L^2_\xi}$.
By construction, we have that
\begin{equation}R(s,\xi)=\frac{i\sqrt{2\pi}}{4\pi}\frac1{s}|\hat{g}(s,\xi)|^2\hat{f}(s,\xi)
-i\mathcal{F}_\xi\left(e^{-is\partial_{xx}}(|e^{is\partial_{xx}}g(s,x)|^2e^{is\partial_{xx}}f(s,x))\right).\label{Rconstruction}\end{equation}
Using that $\|e^{is\partial_{xx}}\varphi\|_{L^2_\xi}=\|\varphi\|_{L^2_\xi}$ and the conservation of the $L^2$ norm of Fourier's transform, we get
$$\|R(s)\|_{L^2_\xi}\lesssim\frac1{s}\|\hat{g}(s)\|^2_{L^\infty_\xi}\|\hat{f}(s)\|_{L^2_\xi}
+\|e^{is\partial_{xx}}g(s)\|_{L^\infty_x}^2\|f(s)\|_{L^2_x}.$$
Thus, we use the Lemma \ref{lem} with $\alpha<\beta$ and the fact that $(a+b)^2\lesssim a^2+b^2$,
$$\begin{array}{rcl}
\|R(s)\|_{L^2_\xi}&\lesssim&s^{-1}\left(\|\hat{g}(s)\|^2_{L^\infty_\xi}\|\hat{f}(s)\|_{L^2_\xi}+
s^{-2\beta}\|\hat{g}(s)\|^2_{H^{1,0}_\xi}\|\hat{f}(s)\|_{L^2_\xi}\right)\\
&\lesssim&s^{-1}\left(\|\hat{w}_g(s)\|^2_{L^\infty_\xi}\|\hat{w}_f(s)\|_{L^2_\xi}+
s^{2(\alpha-\beta)}\|\hat{w}_g(s)\|^2_{X_T}(1+\|\hat{w}_f(s)\|^2_{X_T})^2\|\hat{w}_f(s)\|_{L^2_\xi}\right).
\end{array}$$
Thus we get
\begin{equation}\|R(s)\|_{L^2_\xi}\lesssim s^{-1+\alpha}\|\hat{w}_g\|^2_{X_T}\|\hat{w}_f\|_{X_T}(1+\|\hat{w}_f\|^4_{X_T}),\label{R2}\end{equation}
Now, let us deal with the derivative with respect to $\xi$. In the computation to obtain $\eqref{fh10}$, we already showed that
\begin{equation}\|\partial_\xi B_g(t)\|_{L^2_\xi}\lesssim t^\alpha\|\hat{w}_g\|^2_{X_T}.\label{dxiB2}\end{equation}
For $R$, we go back once again to its construction. We have
\begin{equation}\begin{array}{rcl}
R(s,\xi)&=&-\frac{i}{2\pi}\displaystyle\int e^{2is\eta\sigma}
\hat{g}(s,\xi-\sigma)\overline{\hat{g}}(s,\xi-\eta-\sigma)\hat{f}(s,\xi-\eta)d\sigma d\eta
+\frac{i\sqrt{2\pi}}{4\pi}\frac{1}{s}\left|\hat{g}(s,\xi)\right|^2\hat{f}(s,\xi)\\
&=:&I(\hat{g},\hat{g},\hat{f})(s,\xi)+\dfrac1sN(\hat{g},\hat{g},\hat{f})(s,\xi).\label{RIN}
\end{array}\end{equation}
Then, by linearity, 
\begin{equation}\partial_\xi R=I(\partial_\xi\hat{g},\hat{g},\hat{f})+I(\hat{g},\partial_\xi\hat{g},\hat{f})+I(\hat{g},\hat{g},\partial_\xi\hat{f})+
\frac1sN(\partial_\xi\hat{g},\hat{g},\hat{f})+\frac1sN(\hat{g},\partial_\xi\hat{g},\hat{f})+\frac1sN(\hat{g},\hat{g},\partial_\xi\hat{f}).\label{dxir}\end{equation}                                                                                                                                       
Obviously, the three terms in $I$ and $N$ play the same role. We will focus on one of each.
$$\begin{array}{rcl}
\|I(\partial_\xi\hat{g},\hat{g},\hat{f})(s)\|_{L^2_\xi}&\lesssim&
\|\mathcal{F}_\xi\left(e^{-is\partial_{xx}}(e^{is\partial_{xx}}xg(s,x)e^{-is\partial_{xx}}\overline{g}(s,x)e^{is\partial_{xx}}f(s,x))\right)\|_{L^2_\xi}\\
&\lesssim& \|e^{is\partial_{xx}}f(s)\|_{L^\infty_x}\|e^{is\partial_{xx}}g(s)\|_{L^\infty_x}\|xg(s)\|_{L^2_x}\\
&\lesssim&s^{-1}(\|\hat{f}(s)\|_{L^\infty_\xi}+
s^{-\beta}\|\hat{f}(s)\|_{H^{1,0}_\xi})(\|\hat{g}(s)\|_{L^\infty_\xi}+
s^{-\beta}\|\hat{g}(s)\|_{H^{1,0}_\xi})\|\partial_\xi \hat{g}(s)\|_{L^2_\xi}\\
&\lesssim&s^{-1+\alpha}\left\|\hat{w}_f\right\|_{X_T}(1+\left\|\hat{w}_g\right\|_{X_T}^2)\left\|\hat{w}_g\right\|^2_{X_T}(1+\left\|\hat{w}_f\right\|_{X_T}^4).
\end{array}$$
For $N$, we have
$$\begin{array}{rcl}
\|s^{-1}N(\partial_\xi\hat{g},\hat{g},\hat{f})(s)\|_{L^2_\xi}&\lesssim&s^{-1}\|\hat{f}(s)\|_{L^\infty_\xi}\|\hat{g}(s)\|_{L^\infty_\xi}
\|\partial_{\xi}\hat{g}(s)\|_{L^2_\xi}\\
&\lesssim&s^{-1+\alpha}\left\|\hat{w}_f\right\|_{X_T}\left\|\hat{w}_g\right\|^2_{X_T}(1+\left\|\hat{w}_f\right\|_{X_T}^2).
\end{array}$$
Thus, we can conclude that
\begin{equation}\|\partial_\xi R(s,\xi)\|_{L^2_\xi}\lesssim s^{-1+\alpha}
\left\|\hat{w}_f\right\|_{X_T}\left\|\hat{w}_g\right\|^2_{X_T}(1+\left\|\hat{w}_g\right\|_{X_T}^2)(1+\left\|\hat{w}_f\right\|_{X_T}^4).\label{dxiR2}\end{equation}
The computations are almost the same for the $H^{0,2n+1}_\xi$ part. Since $|B(s,\xi)|=1$, we just focus on $R$. With the decomposition \eqref{RIN}, we have
$$\left|\xi^{2n+1}R(s,\xi)\right|\leq \left|\xi^{2n+1}I(\hat{g},\hat{g},\hat{f})(s,\xi)\right|+\dfrac1s\left|\xi^{2n+1}N(\hat{g},\hat{g},\hat{f})(s,\xi)\right|.$$
For the $I$ term, using that by convexity of $x\mapsto x^{2n+1}$ on $\mathbb{R}^+$, we have 
\begin{equation}\left|\xi\right|^{2n+1}\lesssim\left|\xi-\sigma\right|^{2n+1}+\left|\xi-\eta\right|^{2n+1}+
\left|\xi-\sigma-\eta\right|^{2n+1}.\label{convexity}\end{equation}
We obtain
$$\left|\xi^{2n+1}I(\hat{g},\hat{g},\hat{f})(s,\xi)\right|\leq\left|I(\xi^{2n+1}\hat{g},\hat{g},\hat{f})(s,\xi)\right|+
\left|I(\hat{g},\xi^{2n+1}\hat{g},\hat{f})(s,\xi)\right|+\left|I(\hat{g},\hat{g},\xi^{2n+1}\hat{f})(s,\xi)\right|.$$
Let's deal for example with the third part, we have
$$\begin{array}{rcl}
\|I(\hat{g},\hat{g},\xi^{2n+1}\hat{f})(s,\xi)\|_{L^2_\xi}&\lesssim&\|e^{is\partial_{xx}}g(s)\|^2_{L^\infty_x}\|\partial_x^{2n+1}f(s)\|_{L^2_x}\\
&\lesssim&s^{-1}(\|\hat{g}(s)\|_{L^\infty_\xi}+s^{-\beta}\|\hat{g}(s)\|_{H^{1,0}_\xi})^2\|\xi^{2n+1} \hat{f}(s)\|_{L^2_\xi}\\
&\lesssim&s^{-1+\alpha}\left\|\hat{w}_f\right\|_{X_T}\left\|\hat{w}_g\right\|^2_{X_T}(1+\left\|\hat{w}_f\right\|_{X_T}^4).
\end{array}$$
Doing the same for the two others parts, we conclude for the $I$ part by
$$\|\xi^{2n+1}I(\hat{g},\hat{g},\hat{f})(s)\|_{L^2_\xi}\lesssim
s^{-1+\alpha}\left\|\hat{w}_f\right\|_{X_T}\left\|\hat{w}_g\right\|^2_{X_T}(1+\left\|\hat{w}_g\right\|_{X_T}^2)(1+\left\|\hat{w}_f\right\|_{X_T}^4).$$
The $N$ term is simpler, cause we remark that 
$\left|\xi^{2n+1}N(\hat{g},\hat{g},\hat{f})(s,\xi)\right|=\left|N(\xi^{2n+1}\hat{g},\hat{g},\hat{f})(s,\xi)\right|$. Thus we have that
$$\begin{array}{rcl}
\|s^{-1}\xi^{2n+1}N(\hat{g},\hat{g},\hat{f})(s)\|_{L^2_\xi}&\lesssim&s^{-1}\|\hat{f}(s)\|_{L^\infty_\xi}\|\hat{g}(s)\|_{L^\infty_\xi}
\|\xi^{2n+1}\hat{g}(s)\|_{L^2_\xi}\\
&\lesssim&s^{-1+\alpha}\left\|\hat{w}_f\right\|_{X_T}\left\|\hat{w}_g\right\|^2_{X_T}.
\end{array}$$
Finally we conclude for $R$ that
\begin{equation}\|\xi^{2n+1} R(s)\|_{L^2_\xi}\lesssim s^{-1+\alpha}\|\hat{w}_g\|^2_{X_T}\|\hat{w}_f\|_{X_T}(1+\left\|\hat{w}_g\right\|_{X_T}^2)
(1+\|\hat{w}_f\|^4_{X_T}).\label{xiR2}\end{equation}

\subsection{Conclusion of the proof.} 
We have to compile our estimates with the norms that defined the $X_T$ norm.
We remind that $$\|\hat{w}_f\|_{X_T}:=\|\hat{w}_f\|_{L^\infty_TL^\infty_\xi}+\|t^{-\alpha}\hat{w}_f\|_{L^\infty_TH^{1,0}_\xi}
+\|t^{-\alpha}\hat{w}_f\|_{L^\infty_TH^{0,2n+1}_\xi},$$
and $$\Phi_1(\hat{w}_f,\hat{w}_g)(t,\xi):=\hat{w}_f^1(\xi)+\displaystyle\int_1^tB_g(s,\xi)R(s,\xi)ds.$$
For the $L^{\infty}$ part, we use the Lemma \ref{lemR}, the equation \eqref{fh10} and the fact that we have $3\alpha-\delta<0$,
$$\begin{array}{rcl}  
\|\Phi_1(\hat{w}_f,\hat{w}_g)(t)\|_{L^\infty_\xi}&\leq&\|\hat{w}_f^1\|_{L^\infty_\xi}+\displaystyle\int_1^t\|R(s)\|_{L^\infty_\xi}ds\\
&\lesssim&\varepsilon+\left\|\hat{w}_f\right\|_{X_T}\left\|\hat{w}_g\right\|_{X_T}^2(1+\left\|\hat{w}_g\right\|_{X_T}^2)
(1+\left\|\hat{w}_f\right\|_{X_T}^4).
\end{array}$$
For the $H^{1,0}$ part, we use the Lemma \ref{lemR}, the equations \eqref{R2}, \eqref{dxiB2} and \eqref{dxiR2},
$$\begin{array}{rcl}  
\|t^{-\alpha}\Phi_1(\hat{w}_f,\hat{w}_g)(t)\|_{H^{1,0}_\xi}&\leq&\|\hat{w}_f^1\|_{H^{1,0}_\xi}+t^{-\alpha}\displaystyle\int_1^t
\left(\|R(s)\|_{H^{1,0}_\xi}+\|R(s)\|_{L^\infty_\xi}\|\partial_\xi B_g(s)\|_{L^2_\xi}\right)ds\\
&\lesssim&\varepsilon+\left\|\hat{w}_f\right\|_{X_T}\left\|\hat{w}_g\right\|_{X_T}^2(1+\left\|\hat{w}_f\right\|_{X_T}^4)
(1+\left\|\hat{w}_g\right\|_{X_T}^4).
\end{array}$$
The $H^{0,2n+1}$ part is easier, just using \eqref{xiR2} we get
$$\|t^{-\alpha}\Phi_1(\hat{w}_f,\hat{w}_g)(t)\|_{H^{0,2n+1}_\xi}\lesssim
\varepsilon+\left\|\hat{w}_f\right\|_{X_T}\left\|\hat{w}_g\right\|_{X_T}^2(1+\left\|\hat{w}_f\right\|_{X_T}^4)(1+\left\|\hat{w}_g\right\|_{X_T}^2).$$
Finally, we have, with a constant $M>0$,
$$\left\|\Phi_1(\hat{w}_f,\hat{w}_g)\right\|_{X_T}\leq\varepsilon+M\left\|\hat{w}_f\right\|_{X_T}\left\|\hat{w}_g\right\|_{X_T}^2(1+\left\|\hat{w}_f\right\|_{X_T}^4)
(1+\left\|\hat{w}_g\right\|_{X_T}^4).$$
The other equation is obtained by the same way by symmetry of the roles of $u$ and $v$.
\qed

\section{Proof of the Theorem \ref{theo}}\label{sec4}

We are now able to use the proposition in order to prove the Theorem \ref{theo}. First, this gives us a local existence, by a fixed point argument in a 
suitable space. Moreover, by a continuity argument, we prove a global one.
In the two last parts of this section, we use our estimates to get the $L^\infty$ decay and the scattering.

\subsection{Local existence and uniqueness of the solution.}
For the introduction of the function $\Phi$, we need to take care of the notations while considering two pairs of solutions.
We consider the application $\Phi:(\hat{w}_f,\hat{w}_g)\mapsto\left(\Phi_1(\hat{w}_f,\hat{w}_g),\Phi_2(\hat{w}_f,\hat{w}_g)\right)$ with $\Phi_1$ and $\Phi_2$ defined in \eqref{Phi}.
The idea is to apply a fixed point argument to~$\Phi$ in 
$$E:=\left\{(\hat{w}_f,\hat{w}_g)\in X_T, \|\hat{w}_f\|_{X_T}<2\varepsilon, \|\hat{w}_g\|_{X_T}<2\varepsilon \right\},$$ 
that we endow with the norm $\|(\hat{w}_f,\hat{w}_g)\|_E=\|\hat{w}_f\|_{X_T}+\|\hat{w}_g\|_{X_T}$. \\
Using the Proposition \ref{pro}, we have that 
\begin{equation*}\begin{cases}
\left\|\Phi_1(\hat{w}_f,\hat{w}_g)\right\|_{X_T}&\leq\varepsilon+M\left\|\hat{w}_f\right\|_{X_T}\left\|\hat{w}_g\right\|_{X_T}^2(1+\left\|\hat{w}_f\right\|_{X_T}^4)
(1+\left\|\hat{w}_g\right\|_{X_T}^4),\\
\left\|\Phi_2(\hat{w}_f,\hat{w}_g)\right\|_{X_T}&\leq\varepsilon+M\left\|\hat{w}_g\right\|_{X_T}\left\|\hat{w}_f\right\|_{X_T}^2(1+\left\|\hat{w}_g\right\|_{X_T}^4)
(1+\left\|\hat{w}_f\right\|_{X_T}^4).
\end{cases}\end{equation*}
For $\varepsilon$ small enough, 
\begin{equation}8M\varepsilon^3(1+16\varepsilon^4)^2\leq\frac12\varepsilon\label{epsilon}.\end{equation}
In this case, $E$ is invariant under the action of $\Phi$. For this invariance, we could have chosen just $\varepsilon$ instead of $\frac12\varepsilon$
in the RHS of \eqref{epsilon}. The result would have been the same, but we will need a deeper accuracy for the global existence in the next part.
We now have to show that $\Phi$ is a contraction, therefore we consider two pairs of solutions.
Let $(\hat{w}_{f_1},\hat{w}_{g_1}),(\hat{w}_{f_2},\hat{w}_{g_2})\in E$,
$$\|\Phi(\hat{w}_{f_1},\hat{w}_{g_1})-\Phi(\hat{w}_{f_2},\hat{w}_{g_2})\|_E=
\|\Phi_1(\hat{w}_{f_1},\hat{w}_{g_1})-\Phi_1(\hat{w}_{f_2},\hat{w}_{g_2})\|_{X_T}+\|\Phi_2(\hat{w}_{f_1},\hat{w}_{g_1})-
\Phi_2(\hat{w}_{f_2},\hat{w}_{g_2})\|_{X_T}.$$
We do the computations for $\Phi_1$, they are \textit{mutatis mutandis} the same for $\Phi_2$. 
$$\begin{array}{rcl}
\left(\Phi_1(\hat{w}_{f_1},\hat{w}_{g_1})-\Phi_1(\hat{w}_{f_2},\hat{w}_{g_2})\right)(t,\xi)&=&
\displaystyle\int_1^tB_{g_1}(s,\xi)R_1(s,\xi)ds-\displaystyle\int_1^tB_{g_2}(s,\xi)R_2(s,\xi)ds\\
&=&\displaystyle\int_1^tB_{g_1}(R_1-R_2)(s,\xi)ds+\displaystyle\int_1^tR_2(B_{g_1}-B_{g_2})(s,\xi)ds.
\end{array}$$
As we saw in the proof of the Proposition \ref{pro}, the role of $\|\hat{f_1}(t)-\hat{f_2}(t)\|_{H^{1,0}_\xi}$ will be predominant, but we have to do some 
preliminary computations before. We first focus on the coupling effect of $B_{g_1}-B_{g_2}$, which give us in a second time estimates for 
$\hat{f}_1-\hat{f}_2$. In the third part of the proof, we will focus on the effect of the $R_1-R_2$ terms. 
To be as clear as possible, we split the proof by announcing the estimates we are going to prove.

 \subsubsection{The \texorpdfstring{$B_{g_1}-B_{g_2}$}{Lg} estimates.}

First of all, let us rewrite $B_{g_1}-B_{g_2}$. \\ 
For $a,b\in\mathbb{R}$, $e^{ia}-e^{ib}=2i\sin\left(\frac{a-b}{2}\right)e^{i\left(\frac{a+b}2\right)},$ 
so we get
$$(B_{g_1}-B_{g_2})(t,\xi)=2i\sin\left(\displaystyle\int_1^t\frac{|\hat{w}_{g_1}(s,\xi)|^2-|\hat{w}_{g_2}(s,\xi)|^2}{2s}ds\right)
e^{i\left(\displaystyle\int_1^t\frac{|\hat{w}_{g_1}(s,\xi)|^2+|\hat{w}_{g_2}(s,\xi)|^2}{2s}ds\right)}.$$
Thus, we have the following relations:
\begin{itemize}
 \item $\|(B_{g_1}-B_{g_2})(t)\|_{L^{\infty}_\xi}\lesssim\ln(t)\varepsilon\|\hat{w}_{g_1}-\hat{w}_{g_2}\|_{X_T}.$ \\
Actually, we have
$$\begin{array}{rcl}
\|(B_{g_1}-B_{g_2})(t)\|_{L^{\infty}_\xi}&\lesssim&\|\displaystyle\int_1^ts^{-1}(|\hat{w}_{g_1}(s)|^2-|\hat{w}_{g_2}(s)|^2)ds
\|_{L^\infty_\xi}\\
&\lesssim&\displaystyle\int_1^ts^{-1}\|\hat{w}_{g_1}(s)-\hat{w}_{g_2}(s)\|_{L^\infty_\xi}
(\|\hat{w}_{g_1}(s)\|_{L^\infty_\xi}+\|\hat{w}_{g_2}(s)\|_{L^\infty_\xi})ds.
\end{array}$$
 \item $\|\partial_\xi(B_{g_1}-B_{g_2})(t)\|_{L^2_\xi}\lesssim t^\alpha\varepsilon(1+\varepsilon^2\ln(t))\|\hat{w}_{g_1}-\hat{w}_{g_2}\|_{X_T}.$\\
For this relation, we need more computations due to the derivative. We have
$$\begin{array}{rcl}
&&\|\partial_\xi(B_{g_1}-B_{g_2})(t)\|_{L^2_\xi}\leq\|\partial_\xi\sin\left(\displaystyle\int_1^t\frac{|\hat{w}_{g_1}(s)|^2-|\hat{w}_{g_2}(s)|^2}{2s}ds\right)
\|_{L^2_\xi}\\
&&+\|\sin\left(\displaystyle\int_1^t\frac{|\hat{w}_{g_1}(s)|^2-|\hat{w}_{g_2}(s)|^2}{2s}ds\right)\partial_\xi
e^{i\left(\displaystyle\int_1^t\frac{|\hat{w}_{g_1}(s)|^2+|\hat{w}_{g_2}(s)|^2}{2s}ds\right)}\|_{L^2_\xi}.
\end{array}$$
We consider separately the two parts of the sum. For the first part,
$$\begin{array}{rcl}
I_1&=&\|\partial_\xi\sin\left(\displaystyle\int_1^t\frac{|\hat{w}_{g_1}(s)|^2-|\hat{w}_{g_2}(s)|^2}{2s}ds\right)
\|_{L^2_\xi}\\
&\lesssim&\|\displaystyle\int_1^ts^{-1}\partial_\xi(|\hat{w}_{g_1}(s)|^2-|\hat{w}_{g_2}(s)|^2)ds\|_{L^2_\xi}\\
&\lesssim&\displaystyle\int_1^ts^{-1}\left((\|\partial_\xi\hat{w}_{g_1}(s)\|_{L^2_\xi}+\|\partial_\xi\hat{w}_{g_2}(s)\|_{L^2_\xi})
\|\hat{w}_{g_1}(s)-\hat{w}_{g_2}(s)\|_{L^\infty_\xi}\right.\\
&&\left.+\quad(\|\hat{w}_{g_1}(s)\|_{L^\infty_\xi}+\|\hat{w}_{g_2}(s)\|_{L^\infty_\xi})
\|\partial_\xi(\hat{w}_{g_1}(s)-\hat{w}_{g_2}(s))\|_{L^2_\xi}\right)ds\\
&\lesssim&t^\alpha\varepsilon\|\hat{w}_{g_1}-\hat{w}_{g_2}\|_{X_T}.
\end{array}$$
For the second one, using the same relation for the coupled part, 
$$\begin{array}{rcl}
I_2&=&\|\sin\left(\displaystyle\int_1^t\frac{|\hat{w}_{g_1}(s)|^2-|\hat{w}_{g_2}(s)|^2}{2s}ds\right)\partial_\xi
e^{i\left(\displaystyle\int_1^t\frac{|\hat{w}_{g_1}(s)|^2+|\hat{w}_{g_2}(s)|^2}{2s}ds\right)}\|_{L^2_\xi}\\
&\lesssim&\displaystyle\int_1^ts^{-1}\|\hat{w}_{g_1}(s)-\hat{w}_{g_2}(s)\|_{L^\infty_\xi}
(\|\hat{w}_{g_1}(s)\|_{L^\infty_\xi}+\|\hat{w}_{g_2}(s)\|_{L^\infty_\xi})ds\\
&&\times \displaystyle\int_1^ts^{-1}(\|\partial_\xi\hat{w}_{g_1}(s)\|_{L^2_\xi}\|\hat{w}_{g_1}(s)\|_{L^\infty_\xi}
+\|\partial_\xi\hat{w}_{g_2}(s)\|_{L^2_\xi}\|\hat{w}_{g_2}(s)\|_{L^\infty_\xi})ds\\
&\lesssim&t^\alpha\ln(t)\varepsilon^3\|\hat{w}_{g_1}-\hat{w}_{g_2}\|_{X_T}.
\end{array}$$
\end{itemize}

\subsubsection{The \texorpdfstring{$\hat{f}_1-\hat{f}_2$}{Lg} estimates.}

Using the previous estimates, we now deal with the coupling effect of $\hat{f}_1-\hat{f}_2$. First, we see that
$$\hat{f}_1-\hat{f}_2=\hat{w}_{f_1}\overline{B}_{g_1}-\hat{w}_{f_2}\overline{B}_{g_2}=
\hat{w}_{f_1}(\overline{B}_{g_1}-\overline{B}_{g_2})+\overline{B}_{g_2}(\hat{w}_{f_1}-\hat{w}_{f_2}).$$
Thanks to this decomposition, we deal with the $L^\infty$, $H^{1,0}$ and $H^{0,2n+1}$ norms of $\hat{f}_1-\hat{f}_2$.
\begin{itemize}
 \item $\|(\hat{f}_1-\hat{f}_2)(t)\|_{L^\infty_\xi}\lesssim
\ln(t)\varepsilon^2\|\hat{w}_{g_1}-\hat{w}_{g_2}\|_{X_T}+\|\hat{w}_{f_1}-\hat{w}_{f_2}\|_{X_T}.$\\
We prove this inequality by using the estimation of $\|(B_{g_1}-B_{g_2})(t)\|_{L^{\infty}_\xi}$. We have
$$\begin{array}{rcl}
\|(\hat{f}_1-\hat{f}_2)(t)\|_{L^\infty_\xi}&\lesssim&\|\hat{w}_{f_1}(t)\|_{L^\infty_\xi}\|(B_{g_1}-B_{g_2})(t)\|_{L^{\infty}_\xi}
+\|(\hat{w}_{f_1}-\hat{w}_{f_2})(t)\|_{L^\infty_\xi}.
\end{array}$$
 \item $\|(\hat{f}_1-\hat{f}_2)(t)\|_{L^2_\xi}\lesssim
t^\alpha\ln(t)\varepsilon^2\|\hat{w}_{g_1}-\hat{w}_{g_2}\|_{X_T}+t^\alpha\|\hat{w}_{f_1}-\hat{w}_{f_2}\|_{X_T}.$\\
Indeed, by the exact same way but with the $L^2$ control of the $\hat{w}$ terms,
$$\begin{array}{rcl}
\|(\hat{f}_1-\hat{f}_2)(t)\|_{L^2_\xi}&\lesssim&\|\hat{w}_{f_1}(t)\|_{L^2_\xi}\|(B_{g_1}-B_{g_2})(t)\|_{L^{\infty}_\xi}
+\|(\hat{w}_{f_1}-\hat{w}_{f_2})(t)\|_{L^2_\xi}.
\end{array}$$
  \item $\|\xi^{2n+1}(\hat{f}_1-\hat{f}_2)(t)\|_{L^2_\xi}\lesssim
t^\alpha\ln(t)\varepsilon^2\|\hat{w}_{g_1}-\hat{w}_{g_2}\|_{X_T}+t^\alpha\|\hat{w}_{f_1}-\hat{w}_{f_2}\|_{X_T}.$\\
This is the same estimation and exactly the same computation because 
$$\begin{array}{rcl}
\|\xi^{2n+1}(\hat{f}_1-\hat{f}_2)(t)\|_{L^2_\xi}&\lesssim&\|\xi^{2n+1}\hat{w}_{f_1}(t)\|_{L^2_\xi}\|(B_{g_1}-B_{g_2})(t)\|_{L^{\infty}_\xi}
+\|\xi^{2n+1}(\hat{w}_{f_1}-\hat{w}_{f_2})(t)\|_{L^2_\xi}.
\end{array}$$
 \item $\|\partial_\xi(\hat{f}_1-\hat{f}_2)(t)\|_{L^2_\xi}\lesssim
t^\alpha\ln(t)\varepsilon^2\|\hat{w}_{g_1}-\hat{w}_{g_2}\|_{X_T}+t^\alpha\|\hat{w}_{f_1}-\hat{w}_{f_2}\|_{X_T}.$\\
For this one we have
$$\begin{array}{rcl}\|\partial_\xi(\hat{f}_1-\hat{f}_2)(t)\|_{L^2_\xi}&\lesssim&\|\partial_\xi\hat{w}_{f_1}(t)\|_{L^2_\xi}\|(B_{g_1}-B_{g_2})(t)\|_{L^{\infty}_\xi}
+\|\hat{w}_{f_1}(t)\|_{L^{\infty}_\xi}\|\partial_\xi(B_{g_1}-B_{g_2})(t)\|_{L^2_\xi}\\
&&+\:\|\partial_\xi B_{g_2}(t)\|_{L^2_\xi}\|(\hat{w}_{f_1}-\hat{w}_{f_2})(t)\|_{L^{\infty}_\xi}
+\|\partial_\xi(\hat{w}_{f_1}-\hat{w}_{f_2})(t)\|_{L^2_\xi}.\\
\end{array}$$
Here, we can use the previous relations of this part, the equation \eqref{dxiB2}, and the fact that $\varepsilon$ is small to obtain the desired estimate.
\end{itemize}
Computing these estimates, we finally have
\begin{equation}\begin{array}{rcl}
\|(\hat{f}_1-\hat{f}_2)(t)\|_{L^\infty_\xi}&\lesssim&
\ln(t)\varepsilon^2\|\hat{w}_{g_1}-\hat{w}_{g_2}\|_{X_T}+\|\hat{w}_{f_1}-\hat{w}_{f_2}\|_{X_T},\\
\|(\hat{f}_1-\hat{f}_2)(t)\|_{H^{1,0}_\xi}&\lesssim&
t^\alpha\ln(t)\varepsilon^2\|\hat{w}_{g_1}-\hat{w}_{g_2}\|_{X_T}+t^\alpha\|\hat{w}_{f_1}-\hat{w}_{f_2}\|_{X_T},\\
\|(\hat{f}_1-\hat{f}_2)(t)\|_{H^{0,2n+1}_\xi}&\lesssim&
t^\alpha\ln(t)\varepsilon^2\|\hat{w}_{g_1}-\hat{w}_{g_2}\|_{X_T}+t^\alpha\|\hat{w}_{f_1}-\hat{w}_{f_2}\|_{X_T}.
\end{array}\label{f1-f2}\end{equation}

\subsubsection{The \texorpdfstring{$R_1-R_2$}{Lg} estimates.}

\begin{itemize}
 \item $\|(R_1-R_2)(t)\|_{L^{\infty}_\xi}\lesssim 
 t^{-1-\delta+3\alpha}\varepsilon^2(\ln(t)\varepsilon^2+1)(\|\hat{w}_{g_1}-\hat{w}_{g_2}\|_{X_T}+\|\hat{w}_{f_1}-\hat{w}_{f_2}\|_{X_T}).$\\
Indeed, by definition of $R$ we have
$$\begin{array}{rcl}
|(R_1-R_2)(t,\xi)|&=&|R(\hat{g_1},\hat{g_1},\hat{f_1})(t,\xi)-R(\hat{g_2},\hat{g_2},\hat{f_2})(t,\xi)|\\
&\leq&\left(|R(\hat{g_1}-\hat{g_2},\hat{g_1},\hat{f_1})|+|R(\hat{g_2},\hat{g_1}-\hat{g_2},\hat{f_1})|
+|R(\hat{g_2},\hat{g_2},\hat{f_1}-\hat{f_2})|\right)(t,\xi).
\end{array}$$
Therefore, using the Lemma \ref{lemR} and the equation \eqref{f1-f2}, we obtain the announced result.
 \item $\|(R_1-R_2)(t)\|_{L^2_\xi}\lesssim 
 t^{-1+\alpha}\varepsilon^2(\ln(t)\varepsilon^2+1)(\|\hat{w}_{f_1}-\hat{w}_{f_2}\|_{X_T}+\|\hat{w}_{g_1}-\hat{w}_{g_2}\|_{X_T})$.\\
In order to get this estimate, we use the relation \eqref{RIN},
$$\|(R_1-R_2)(t)\|_{L^2_\xi}\leq\|(I(\hat{g_1},\hat{g_1},\hat{f_1})-I(\hat{g_2},\hat{g_2},\hat{f_2}))(t)\|_{L^2_\xi}
+t^{-1}\|(N(\hat{g_1},\hat{g_1},\hat{f_1})-N(\hat{g_2},\hat{g_2},\hat{f_2}))(t)\|_{L^2_\xi}.$$
Once again, by linearity, we remark that 
$$\begin{array}{rcl}
I(\hat{g_1},\hat{g_1},\hat{f_1})-I(\hat{g_2},\hat{g_2},\hat{f_2})&=&I(\hat{g_1}-\hat{g_2},\hat{g_1},\hat{f_1})+
I(\hat{g_2},\hat{g_1}-\hat{g_2},\hat{f_1})+I(\hat{g_2},\hat{g_2},\hat{f_1}-\hat{f_2}),\\
N(\hat{g_1},\hat{g_1},\hat{f_1})-N(\hat{g_2},\hat{g_2},\hat{f_2})&=&N(\hat{g_1}-\hat{g_2},\hat{g_1},\hat{f_1})+
N(\hat{g_2},\hat{g_1}-\hat{g_2},\hat{f_1})+N(\hat{g_2},\hat{g_2},\hat{f_1}-\hat{f_2}).
\end{array}$$
Here, we just show how to control one term in $I$, and one in $N$, all the others computations are the same.
For $I$ we use the Lemma \ref{lem} and the equations \eqref{fh10} and \eqref{f1-f2},
$$\begin{array}{rcl}
\|I(\hat{g_1}-\hat{g_2},\hat{g_1},\hat{f_1})(t)\|_{L^2_\xi}&\lesssim&\|e^{it\partial_{xx}}g_1(t)\|_{L^\infty_x}
\|e^{it\partial_{xx}}f_1(t)\|_{L^\infty_x}\|(g_1-g_2)(t)\|_{L^2_x}\\
&\lesssim&t^{-1}(\|\hat{g_1}\|_{L^\infty_\xi}+t^{-\beta}\|\hat{g_1}\|_{H^{1,0}_\xi})
(\|\hat{f_1}\|_{L^\infty_\xi}+t^{-\beta}\|\hat{f_1}\|_{H^{1,0}_\xi})\|\hat{g_1}-\hat{g_2}\|_{L^2_\xi}\\
&\lesssim&t^{-1+\alpha}\varepsilon^2\left(\ln(t)\varepsilon^2\|\hat{w}_{f_1}-\hat{w}_{f_2}\|_{X_T}+\|\hat{w}_{g_1}-\hat{w}_{g_2}\|_{X_T}\right).
\end{array}$$
For the $N$ term, by equation \eqref{f1-f2} we have
$$\begin{array}{rcl}
\|N(\hat{g_1}-\hat{g_2},\hat{g_1},\hat{f_1})(t)\|_{L^2_\xi}&=&\|(\hat{g_1}-\hat{g_2})\overline{\hat{g_1}}\hat{f_1}(t)\|_{L^2_\xi}\\
&\lesssim&\|\hat{f_1}\|_{L^\infty_\xi}\|\hat{g_1}\|_{L^\infty_\xi}\|(\hat{g_1}-\hat{g_2})\|_{L^2_\xi}\\
&\lesssim&t^\alpha\varepsilon^2\left(\ln(t)\varepsilon^2\|\hat{w}_{f_1}-\hat{w}_{f_2}\|_{X_T}+\|\hat{w}_{g_1}-\hat{w}_{g_2}\|_{X_T}\right).
\end{array}$$
Combining these results, we have the desired one. For the $H^{1,0}$ part, we want to show that
 \item $\|\partial_\xi(R_1-R_2)(t)\|_{L^2_\xi}\lesssim 
 t^{-1+\alpha}\varepsilon^2(\ln(t)\varepsilon^2+1)(\|\hat{w}_{f_1}-\hat{w}_{f_2}\|_{X_T}+\|\hat{w}_{g_1}-\hat{w}_{g_2}\|_{X_T})$.\\
To prove this one, we write $\partial_\xi(R_1-R_2)$ by using the equation \eqref{dxir}, we obtain twelve terms, six terms in $I$, six in $N$. 
We need to couple these terms, using the same method as previously. For example, for $I$ we have
$$I(\partial_\xi\hat{g_1},\hat{g_1},\hat{f_1})-I(\partial_\xi\hat{g_2},\hat{g_2},\hat{f_2})=I(\partial_\xi(\hat{g_1}-\hat{g_2}),\hat{g_1},\hat{f_1})+
I(\partial_\xi\hat{g_2},\hat{g_1}-\hat{g_2},\hat{f_1})+I(\partial_\xi\hat{g_2},\hat{g_2},\hat{f_1}-\hat{f_2}).$$
Thus, we obtain eighteen terms with a coupled part in each of them. Following the way of the estimation of $\|(R_1-R_2)(t)\|_{L^2_\xi}$, and using
the equations \eqref{fh10} and \eqref{f1-f2}, we finally obtain the result.
The last estimate we deal with is
 \item $\|\xi^{2n+1}(R_1-R_2)(t)\|_{L^2_\xi}\lesssim 
 t^{-1+\alpha}\varepsilon^2(\ln(t)\varepsilon^2+1)(\|\hat{w}_{g_1}-\hat{w}_{g_2}\|_{X_T}+\|\hat{w}_{f_1}-\hat{w}_{f_2}\|_{X_T}).$\\
This is the same estimation as $\|(R_1-R_2)(t)\|_{L^2_\xi}$, and the computations are almost the same too.
As in the previous estimate, we use the following decomposition,
$$|\xi^{2n+1}(R_1-R_2)|\leq|\xi^{2n+1} R(\hat{g_1}-\hat{g_2},\hat{g_1},\hat{f_1})|+|\xi^{2n+1} R(\hat{g_2},\hat{g_1}-\hat{g_2},\hat{f_1})|
+|\xi^{2n+1} R(\hat{g_2},\hat{g_2},\hat{f_1}-\hat{f_2})|.$$
Let's just focus on the first term, using the equation \eqref{RIN}, we have that
$$|\xi^{2n+1} R(\hat{g_1}-\hat{g_2},\hat{g_1},\hat{f_1})|\leq|\xi^{2n+1} I(\hat{g_1}-\hat{g_2},\hat{g_1},\hat{f_1})|
+t^{-1}|\xi^{2n+1} N(\hat{g_1}-\hat{g_2},\hat{g_1},\hat{f_1})|.$$
For the $I$ term, we use a second time the equation \eqref{convexity} to obtain
$$|\xi^{2n+1}I(\hat{g_1}-\hat{g_2},\hat{g_1},\hat{f_1})|\leq|I(\xi^{2n+1}(\hat{g_1}-\hat{g_2}),\hat{g_1},\hat{f_1})|
+|I(\hat{g_1}-\hat{g_2},\xi^{2n+1}\hat{g_1},\hat{f_1})|+|I(\hat{g_1}-\hat{g_2},\hat{g_1},\xi^{2n+1}\hat{f_1})|.$$
Now we just do the same computations that previously. We take the $L^2_x$ part of the norm for the coupling part, and the others ones
in $L^\infty_x$. For $N$, we simply have
$$|\xi^{2n+1} N(\hat{g_1}-\hat{g_2},\hat{g_1},\hat{f_1})|=|N(\hat{g_1}-\hat{g_2},\xi^{2n+1}\hat{g_1},\hat{f_1})|.$$
As we have the relation $|\xi^{2n+1}\hat{g_1}|=|\xi^{2n+1}\hat{w}_{g_1}|$, we have the same bound once again.
\end{itemize}

\subsubsection{Conclusion.}

Set $\psi(t):=\left(\Phi_1(\hat{w}_{f_1},\hat{w}_{g_1})-\Phi_1(\hat{w}_{f_2},\hat{w}_{g_2})\right)(t)$, we remind that
$$\psi(t)=\displaystyle\int_1^tB_{g_1}(R_1-R_2)(s,\xi)ds+\displaystyle\int_1^tR_2(B_{g_1}-B_{g_2})(s,\xi)ds.$$
Thus, we have
$$\begin{array}{rcl}
\|\psi(t)\|_{L^{\infty}_\xi}&\leq&\displaystyle\int_1^t\|(R_1-R_2)(s)\|_{L^{\infty}_\xi}ds
+\displaystyle\int_1^t\|R_2(s)\|_{L^{\infty}_\xi}\|(B_{g_1}-B_{g_2})(s)\|_{L^{\infty}_\xi}ds,\\
\|\psi(t)\|_{L^2_\xi}&\leq&\displaystyle\int_1^t\|(R_1-R_2)(s)\|_{L^2_\xi}ds
+\displaystyle\int_1^t\|R_2(s)\|_{L^2_\xi}\|(B_{g_1}-B_{g_2})(s)\|_{L^{\infty}_\xi}ds,\\
\|\xi^{2n+1}\psi(t)\|_{L^2_\xi}&\leq&\displaystyle\int_1^t\|\xi^{2n+1}(R_1-R_2)(s)\|_{L^2_\xi}ds
+\displaystyle\int_1^t\|\xi^{2n+1}R_2(s)\|_{L^2_\xi}\|(B_{g_1}-B_{g_2})(s)\|_{L^{\infty}_\xi}ds,\\
\|\partial_\xi\psi(t)\|_{L^2_\xi}&\leq&\displaystyle\int_1^t\|\partial_\xi B_{g_1}(s)\|_{L^2_\xi}\|(R_1-R_2)(s)\|_{L^\infty_\xi}ds
+\displaystyle\int_1^t\|\partial_\xi(R_1-R_2)(s)\|_{L^2_\xi}ds\\
&&+\displaystyle\int_1^t\|R_2(s)\|_{L^\infty_\xi}\|\partial_\xi(B_{g_1}-B_{g_2})(s)\|_{L^2_\xi}ds
+\displaystyle\int_1^t\|\partial_\xi R_2(s)\|_{L^2_\xi}\|(B_{g_1}-B_{g_2})(s)\|_{L^{\infty}_\xi}ds.\\
\end{array}$$
Finally, we find, for a constant $M>0$,
$$\|\Phi_1(\hat{w}_{f_1},\hat{w}_{g_1})-\Phi_1(\hat{w}_{f_2},\hat{w}_{g_2})\|_{X_T}\leq M\varepsilon^2(1+\ln(t)\varepsilon^2)
(\|\hat{w}_{f_1}-\hat{w}_{f_2}\|_{X_T}+\|\hat{w}_{g_1}-\hat{w}_{g_2}\|_{X_T}).$$
We chose $\varepsilon$ small enough. Thus, for $T$ sufficiently close to 1, $M\varepsilon^2(1+\varepsilon^2\ln(T))\leq\frac14$, and wet get
\begin{equation}
\|\Phi(\hat{w}_{f_1},\hat{w}_{g_1})-\Phi(\hat{w}_{f_2},\hat{w}_{g_2})\|_E\leq\frac12
\|(\hat{w}_{f_1},\hat{w}_{g_1})-(\hat{w}_{f_2},\hat{w}_{g_2})\|_E.\label{contloc}
\end{equation}
Finally, for such a $T$, $\Phi$ is a contraction. We can apply the fixed point theorem to get the local existence of the solutions. Let's see now of to pass from
the local existence to the global one.

\subsection{Global existence of the solution.}
The key argument for the transition from the local existence to the global one is the continuity. First, we show that 
$t\mapsto\|\hat{w}_f(t)\|_{L^\infty_\xi}+\|t^{-\alpha}\hat{w}_f(t)\|_{H^{1,0}_\xi}+\|t^{-\alpha}\hat{w}_f(t)\|_{H^{0,2n+1}_\xi}$ is continuous. 
Then we will conclude by a connectivity argument.
Let's remark that, for $1<t_1<t_2<T$, with the $T$ defined in the previous section,
$$\hat{w}_f(t_2)-\hat{w}_f(t_1)=\displaystyle\int_{t_1}^{t_2}B_g(s,\xi)R(\hat{g},\hat{g},\hat{f})(s,\xi)ds.$$
By the Lemma \ref{lemR} and the estimate \eqref{fh10}, we have that
\begin{equation}\begin{array}{rcl}
\left|\|\hat{w}_f(t_2)\|_{L^\infty_\xi}-\|\hat{w}_f(t_1)\|_{L^\infty_\xi}\right|&\leq&\|\hat{w}_f(t_2)-\hat{w}_f(t_1)\|_{L^\infty_\xi}\\
&\leq&\displaystyle\int_{t_1}^{t_2}\|R(\hat{g},\hat{g},\hat{f})(s)\|_{L^\infty_\xi}ds\\
&\lesssim&\left|{t_2}^{3\alpha-\delta}-{t_1}^{3\alpha-\delta}\right|\underset{t_1\to t_2}{\longrightarrow}0.\label{wcau}
\end{array}\end{equation}
For the $L^2$ part, we remark that
$$\left|t_2^{-\alpha}\hat{w}_f(t_2)-t_1^{-\alpha}\hat{w}_f(t_1)\right|\leq
\left|t_2^{-\alpha}-t_1^{-\alpha}\right|\left|\hat{w}_f(t_2)\right|+t_1^{-\alpha}\left|\hat{w}_f(t_2)-\hat{w}_f(t_1)\right|.$$
With the Proposition \ref{pro}, we control the $\|\hat{w}_f(t_2)\|_{H^{1,0}_\xi}$ and $\|\hat{w}_f(t_2)\|_{H^{0,2n+1}_\xi}$ terms.
Therefore, using this time the Lemma \ref{lemR} and the estimates \eqref{fh10}, \eqref{R2}, \eqref{dxiB2}, \eqref{dxiR2} and \eqref{xiR2}, we have
$$\begin{array}{rcl}
\|\hat{w}_f(t_2)-\hat{w}_f(t_1)\|_{H^{1,0}_\xi}&\leq&\displaystyle\int_{t_1}^{t_2}(\|R(\hat{g},\hat{g},\hat{f})(s)\|_{H^{1,0}_\xi}
+\|R(\hat{g},\hat{g},\hat{f})(s)\|_{L^\infty_\xi}\|\partial_\xi B(s)\|_{L^2_\xi})ds \\
&\lesssim&\left|{t_2}^{\alpha}-{t_1}^{\alpha}\right|
+\left|{t_2}^{4\alpha-\delta}-{t_1}^{4\alpha-\delta}\right|\underset{t_1\to t_2}{\longrightarrow}0.\\
\|\hat{w}_f(t_2)-\hat{w}_f(t_1)\|_{H^{0,2n+1}_\xi}&\leq&\displaystyle\int_{t_1}^{t_2}\|R(\hat{g},\hat{g},\hat{f})(s)\|_{H^{0,2n+1}_\xi}ds\\
&\lesssim&\left|{t_2}^{\alpha}-{t_1}^{\alpha}\right|\underset{t_1\to t_2}{\longrightarrow}0.
\end{array}$$
Thus, we obtain by the previous computation the continuity of the application 
$$t\mapsto\|\hat{w}_f(t)\|_{L^\infty_\xi}+\|t^{-\alpha}\hat{w}_f(t)\|_{H^{1,0}_\xi}+\|t^{-\alpha}\hat{w}_f(t)\|_{H^{0,2n+1}_\xi}.$$ 
Finally, taking the supremum of this application on $[1,T]$, we have the continuity of the application $T\mapsto\|\hat{w}_f\|_{X_T}$.
Thanks to this continuity, we define the space 
$$A:=\left\{T\geq1, \|\hat{w}_f\|_{X_T}\leq 2\varepsilon\right\}\subset[1,+\infty).$$
Now, we show that $A=[1,+\infty).$ Let's remark that,
\begin{itemize}
 \item $A$ is non empty.\\
 Indeed, by the hypothesis of the Theorem \ref{theo} and the embedding of $H^1(\mathbb{R})$ in $L^\infty(\mathbb{R})$, $1\in A$. 
 \item $A$ is closed.\\
 This is due to the large inequality in the definition of $A$ and the continuity of the application $T\mapsto\|\hat{w}_f\|_{X_T}$.
 \item $A$ is open.\\
 Here we use the Proposition \ref{pro}, if $T\in A$, then we have that :
$$\begin{array}{rcl}
\left\|\hat{w}_f\right\|_{X_T}&\leq&\varepsilon
+M\left\|\hat{w}_f\right\|_{X_T}\left\|\hat{w}_g\right\|_{X_T}^2(1+\left\|\hat{w}_f\right\|_{X_T}^4)(1+\left\|\hat{w}_g\right\|_{X_T}^4)\\
&\leq&\varepsilon+8M\varepsilon^3(1+16\varepsilon^4)^2.
\end{array}$$
Therefore, by the choice of $\varepsilon$ in \eqref{epsilon}, $\left\|\hat{w}_f\right\|_{X_T}\leq\frac32\varepsilon,$  we can go further for $T$. \end{itemize}
Finally, $A$ is non empty, open, closed and included in the interval $[1,+\infty),$ therefore $A=[1,+\infty)$ by connectivity. Thus, we have the global existence
under the assumptions of the Proposition \ref{pro}. 
The last thing to check is that these assumptions are satisfied under the ones of the Theorem \ref{theo}.
Fix $\varepsilon>0$ and suppose that $\left\|u_1\right\|_{H^{1,0}_x}+\left\|u_1\right\|_{H^{0,2n+1}_x}\leq\frac\varepsilon2$.
We want to control $\hat{w}_f^1(\xi)=\hat{f}(1,\xi)=e^{i\xi^2}\hat{u}(1,\xi)$, we have
$$\begin{array}{rcl}
\|\hat{w}_f^1(\xi)\|_{H^{1,0}_\xi}&\leq&\|\hat{u}(1)\|_{L^2_\xi}+\|\xi\hat{u}(1)\|_{L^2_\xi}+\|\partial_\xi\hat{u}(1)\|_{L^2_\xi}\leq\frac\varepsilon2,\\
\|\hat{w}_f^1(\xi)\|_{H^{0,2n+1}_\xi}&=&\|\hat{u}(1)\|_{H^{0,2n+1}_\xi}\leq\frac\varepsilon2.
\end{array}$$
Therefore,
$$\|\hat{w}_f^1(\xi)\|_{H^{1,0}_\xi}+\|\hat{w}_f^1(\xi)\|_{H^{0,2n+1}_\xi}\leq\varepsilon.$$
The hypothesis of the Proposition \ref{pro} are checked, and we have the global solution for the Theorem \ref{theo}.

\subsection{The decay estimate.}
This estimate comes from the Lemma \ref{lem}. Indeed, let $(u,v)$ be the pair of solutions of our problem \eqref{sys}, thus we have that for all $\beta\in(0,\frac{1}{4})$,
\begin{align*}
\|u(t)\|_{L^\infty_x}=\|e^{it\partial_{xx}}f(t)\|_{L^\infty_x}&\lesssim\dfrac{1}{t^{\frac{1}{2}}}\|\hat{f}(t)\|_{L^\infty_\xi}
 +\dfrac{1}{t^{\frac{1}{2}+\beta}}\|\hat{f}(t)\|_{H^{1,0}_\xi}\\
 &\lesssim \dfrac{1}{t^{\frac{1}{2}}}\|\hat{w}_f\|_{L^\infty_\xi}+\dfrac{1}{t^{\frac{1}{2}+\beta}}\|\hat{f}(t)\|_{H^{1,0}_\xi}.\\
\end{align*}
By the definition of the $X_T$ norm and the bound \eqref{fh10}, we get
\begin{equation*}
\|u(t)\|_{L^\infty_x}\lesssim t^{-\frac12}(1+t^{\alpha-\beta})\left\|\hat{w}_f\right\|_{X_T}(1+\left\|\hat{w}_g\right\|_{X_T}^2).
\end{equation*}
Remembering that $\alpha<\beta$, we finally use the control of the $X_T$ norm from the previous subsection to get the desired decay of the Theorem \ref{theo}.

\subsection{The long-range scattering.}\label{subfin}
All the ingredients we need to exhibit the long-range scattering are already in the equation \eqref{wcau},
$$\|\hat{w}_f(t_2)-\hat{w}_f(t_1)\|_{L^\infty_\xi}\lesssim\left|{t_2}^{3\alpha-\delta}-{t_1}^{3\alpha-\delta}\right|.$$
As $3\alpha-\delta<0$, the Cauchy criterion for the existence of a limit of a function allows us to construct the application 
$W^\infty_f\in L^\infty_\xi(\mathbb{R})$ by $W^\infty_f:\xi\mapsto\lim\limits_{t \rightarrow +\infty}\hat{w}_f(t,\xi)$ in the $L^\infty$ sense.
Let $t_1\rightarrow+\infty$ in the previous equation, we obtain
\begin{equation}
\|\hat{w}_f(t)-W^\infty_f\|_{L^\infty_\xi}\lesssim t^{3\alpha-\delta}.\label{W-w}
\end{equation}
We still have to deal with the $H^{0,n}$ case. For that purpose, we'll use the following embedding
\begin{equation}\|f\|_{L^1_\xi}\lesssim\|f\|_{L^2_\xi}^{\frac12}\|xf\|_{L^2_\xi}^{\frac12}.\label{fl1} \end{equation}
Let's prove this inequality, using Cauchy-Schwarz. Let $M>0$,
$$\begin{array}{rcl}
\|f\|_{L^1_x}&=&\displaystyle\int_{\left|x\right|<M}1\left|f\right|+\displaystyle\int_{\left|x\right|>M}x^{-1}x\left|f\right|\\
&\lesssim&M^\frac12\|f\|_{L^2_x}+M^{-\frac12}\|xf\|_{L^2_x}.
\end{array}$$
Choosing $M=\|xf\|_{L^2_x}^{\frac12}\|f\|_{L^2_x}^{-\frac12}$, we get the inequality \eqref{fl1}.
Thereby, we have that
\begin{equation}\begin{array}{rcl}
\|\xi^nR\|_{L^2_\xi}&\lesssim&\|R^\frac12\|_{L^\infty_\xi}\|\xi^nR^\frac12\|_{L^2_\xi}\\
&\lesssim&\|R\|^\frac12_{L^\infty_\xi}\|\xi^{2n}R\|^\frac12_{L^1_\xi}\\
&\lesssim&\|R\|^\frac12_{L^\infty_\xi}\|R\|^\frac12_{H^{0,2n+1}_\xi},
\end{array}\label{perte} \end{equation}
where we used \eqref{fl1} in the last inequality.
Finally we have 
$$\|\hat{w}_f(t_2)-\hat{w}_f(t_1)\|_{H^{0,n}_\xi}\lesssim\left|{t_2}^{\frac{4\alpha-\delta}2}-{t_1}^{\frac{4\alpha-\delta}2}\right|.$$
As $4\alpha-\delta<0$ (see Remark \ref{rem4}), the Cauchy criterion for the existence of a limit of a function allows us to construct the application 
$W^n_f\in H^{0,n}_\xi(\mathbb{R})$ by $W^n_f:\xi\mapsto\lim\limits_{t \rightarrow +\infty}\hat{w}_f(t,\xi)$ in the $H^{0,n}$ sense.
Let $t_1\rightarrow+\infty$ in the previous equation, we obtain
$$\|\hat{w}_f(t)-W^n_f\|_{H^{0,n}_\xi}\lesssim t^{\frac{4\alpha-\delta}2}.$$
These inequalities are verified for all $\alpha$ small enough and all $0<\delta<\frac{1}{4}$. The idea is to take the limits for~$\alpha\rightarrow0$ and 
$\delta\rightarrow\frac14$. Let $\nu:=\frac14-\delta+4\alpha\in(0,\frac14)$, we obtain the desired estimates of long-range scattering.  
The last thing to check is that $W^\infty_f=W^n_f$.
This is a consequence of the Riesz-Fischer Theorem for the completeness of $L^p$. Indeed, it's proved in this theorem that if $\hat{w}_f(t_n)$ converges to 
$W$ in $L^p$, then there exists a subsequence $\hat{w}_f(t_{\varphi(n)})$ that converges to $W$ almost everywhere. Taking $p=2$ and $p=\infty$, we show that 
$W^\infty_f=W^n_f$ almost everywhere. Thus we can define $W_f:=W^\infty_f=W^n_f\in L^\infty_\xi\cap H^{0,n}$.

\subsection{The asymptotic formula.}
From the previous part, we know that it exists an unique $W^\infty_f\in L^\infty_\xi(\mathbb{R})$ such that
\begin{equation*}
\|\hat{f}(t,.)\exp\left(\dfrac{i}{2\sqrt{2\pi}}\displaystyle\int_1^t\frac{1}{s}\left|\hat{v}(s,.)\right|^2ds\right)-W_f\|_{L^\infty_\xi}\lesssim t^{-\frac14+\nu}. 
\end{equation*}
By the equation \eqref{defA} we have
\begin{align*}
u(t,x)=e^{it\partial_{xx}}f(t,x)=(2it)^{-\frac12}e^{i\frac{x^2}{4t}}\hat{f}(\frac{x}{2t})+A_f(t,x), 
\end{align*}
where, using the bound \eqref{fh10}, 
\begin{equation*}
 \left\|A_f(t)\right\|_{L^\infty_x}\lesssim t^{-\frac12-\delta}\|\hat{f}\|_{H^{1,0}_\xi}\lesssim t^{-\frac12+\alpha-\delta}\lesssim t^{-\frac34+\nu}.
\end{equation*}
Thus, we obtain the following asymptotic expansion for $u$, for a large time $t$:
\begin{equation}
u(t,x)=\dfrac1{(2it)^{\frac12}}\exp(i\dfrac{x^2}{4t})W_f(\dfrac x{2t})\exp(\dfrac{-i}{2\sqrt{2\pi}}\displaystyle\int_1^t\frac1\tau|\hat{v}(\tau,\frac x{2t})|^2d\tau)
+\mathcal{O}(t^{-\frac34+\nu}). 
\label{asymp1}\end{equation}
This is an implicit formula. Indeed we have a $v$ term in the RHS of the equality. To deal with this term, we follow the method of Hayashi and Naumkin in \cite{HN}.
Let us define 
\begin{equation}
\gamma_g(t):=\displaystyle\int_1^t(|\hat{w}_g(\tau)|^2-|\hat{w}_g(t)|^2)\dfrac{d\tau}\tau.
\label{defgamma}\end{equation}
Thus, for $1\leq s\leq t$, we have
\begin{equation*}
\gamma_g(t)-\gamma_g(s)=\displaystyle\int_s^t(|\hat{w}_g(\tau)|^2-|\hat{w}_g(t)|^2)\dfrac{d\tau}\tau+(|\hat{w}_g(s)|^2-|\hat{w}_g(t)|^2)\ln(s). 
\end{equation*}
We use the bounds of the Lemma \ref{lemR} and the equation \eqref{fh10} to get
\begin{equation*}
\left||\hat{w}_g(\tau)|^2-|\hat{w}_g(t)|^2\right|\lesssim\left||\hat{w}_g(\tau)|-|\hat{w}_g(t)|\right|\lesssim\left|\tau^{3\alpha-\delta}-t^{3\alpha-\delta}\right|.
\end{equation*}
By using this bound in the previous equation we obtain
\begin{equation*}
|\gamma_g(t)-\gamma_g(s)|\lesssim\left|t^{3\alpha-\delta}-s^{3\alpha-\delta}\right|+\left|t^{3\alpha-\delta}\ln(t)-s^{3\alpha-\delta}\ln(s)\right|. 
\end{equation*}
Therefore, by the Cauchy criterion, it exists $\Gamma_g\in L^\infty_\xi(\mathbb{R})$ defined by 
$\Gamma_g:\xi\mapsto\lim\limits_{t \rightarrow +\infty}\gamma_g(t,\xi)$ in $L^\infty_\xi$.
Let~$t\rightarrow+\infty$ in the previous equation. We use that $\ln(t)\lesssim t^\alpha$ to obtain
\begin{equation}
|\gamma_g(s)-\Gamma_g|\lesssim s^{4\alpha-\delta}.
\label{Gamma-gamma}\end{equation}
Back to the definition of $\gamma_g$ in \eqref{defgamma}, we have 
\begin{equation*}\begin{array}{rcl}
 \displaystyle\int_1^t|\hat{w}_g(\tau)|^2\dfrac{d\tau}\tau&=&\gamma_g(t)+|\hat{w}_g(t)|^2\ln(t)\\
 &=&|W_g|^2\ln(t)+\Gamma_g\\
 &&+\left(\gamma_g(t)-\Gamma_g\right)+\left(|\hat{w}_g(t)|^2-|W_g|^2\right)\ln(t).
\end{array}\end{equation*}
Thus, using the equations \eqref{W-w} and \eqref{Gamma-gamma}, we obtain
\begin{equation*}
\|\displaystyle\int_1^t|\hat{w}_g(\tau)|^2\dfrac{d\tau}\tau-|W_g|^2\ln(t)-\Gamma_g\|_{L^\infty_\xi}\lesssim t^{4\alpha-\delta}\lesssim t^{\nu-\frac14}.
\end{equation*}
We combine this equation with the equation \eqref{asymp1} to get the formula for $u$ for large time $t$.

\bibliographystyle{abbrv}
\bibliography{bibli}

\begin{thebibliography}{10}

\bibitem{AB}
G.~Agrawal and R.~Boyd.
\newblock Contemporary nonlinear optics.
\newblock pages 60--67, 1992.

\bibitem{CKSTT}
J.~Colliander, M.~Keel, G.~Staffilani, H.~Takaoka, and T.~Tao.
\newblock Transfer of energy to high frequencies in the cubic defocusing
  nonlinear {S}chr\"odinger equation.
\newblock {\em Invent. Math.}, 181(1):39--113, 2010.

\bibitem{GPT}
B.~Gr{\'e}bert, {\'E}.~Paturel, and L.~Thomann.
\newblock Beating effects in cubic {S}chr\"odinger systems and growth of
  {S}obolev norms.
\newblock {\em Nonlinearity}, 26(5):1361--1376, 2013.

\bibitem{GPT2}
B.~Gr{\'e}bert, E.~Paturel, and L.~Thomann.
\newblock Modified scattering for the cubic {S}chr\"odinger equation on product
  spaces: the nonresonant case.
\newblock {\em to appear in Math. Res. Lett., arXiv:1502.07699}, 2015.

\bibitem{HPTV}
Z.~Hani, B.~Pausader, N.~Tzvetkov, and N.~Visciglia.
\newblock Modified scattering for the cubic {S}chr\"odinger equation on product
  spaces and applications.
\newblock {\em arXiv:1311.2275}, 2013.

\bibitem{HT}
Z.~Hani and L.~Thomann.
\newblock Asymptotic behavior of the nonlinear {S}chr\"odinger equation with
  harmonic trapping.
\newblock {\em to appear in Comm. Pure Appl. Math., arXiv:1408.6213}, 2014.

\bibitem{HN}
N.~Hayashi and P.~I. Naumkin.
\newblock Asymptotics for large time of solutions to the nonlinear
  {S}chr\"odinger and {H}artree equations.
\newblock {\em Amer. J. Math.}, 120(2):369--389, 1998.

\bibitem{KP}
J.~Kato and F.~Pusateri.
\newblock A new proof of long-range scattering for critical nonlinear
  {S}chr\"odinger equations.
\newblock {\em Differential Integral Equations}, 24(9-10):923--940, 2011.

\bibitem{Kim}
D.~Kim.
\newblock A note on decay rates of solutions to a system of cubic nonlinear
  {S}chr\"odinger equations in one space dimension.
\newblock {\em arXiv:1408.6464}, 2014.

\bibitem{TV}
N.~Tzvetkov and N.~Visciglia.
\newblock Small data scattering for the nonlinear {S}chr\"odinger equation on
  product spaces.
\newblock {\em Comm. Partial Differential Equations}, 37(1):125--135, 2012.

\bibitem{TV2}
N.~Tzvetkov and N.~Visciglia.
\newblock Well-posedness and scattering for nls on
  $\mathbb{R}^d\times\mathbb{T}$ in the energy space.
\newblock {\em arXiv:1409.3938}, page~3, 2014.

\bibitem{Xu}
H.~Xu.
\newblock Unbounded sobolev trajectories and modified scattering theory for a
  wave guide nonlinear {S}chr\"odinger equation.
\newblock {\em arXiv:1506.07350}, 2015.

\end{thebibliography}

\end{document}